\documentclass[10pt,twoside]{amsart}
\usepackage{amsmath,amsfonts,amssymb,amsxtra,bm,setspace,xspace,graphicx,lmodern,psfrag,epsfig,color,latexsym,stmaryrd,scalerel,stackengine,url,enumitem}
\usepackage[colorlinks=true]{hyperref}\hypersetup{urlcolor=blue, citecolor=red}

\numberwithin{equation}{section}
\newtheorem{thm}{Theorem}[section]

\newtheorem{lem}[thm]{Lemma}

\newtheorem{rem}[thm]{Remark}

\newtheorem{exm}[thm]{Example}

\stackMath
\newcommand\reallywidehat[1]{%
	\savestack{\tmpbox}{\stretchto{%
			\scaleto{%
				\scalerel*[\widthof{\ensuremath{#1}}]{\kern-.6pt\bigwedge\kern-.6pt}%
				{\rule[-\textheight/2]{1ex}{\textheight}}
			}{\textheight}%
		}{0.5ex}}%
	\stackon[1pt]{#1}{\tmpbox}%
}
\newcommand{\vertiii}[1]{{\left\vert\kern-0.25ex\left\vert\kern-0.25ex\left\vert #1 
		\right\vert\kern-0.25ex\right\vert\kern-0.25ex\right\vert}}

\newcommand{\R}{{\mathbb R}}

\newcommand{\N}{{\mathbb N}}

\begin{document}
	
	\title[DG method for Vlasov-Stokes]{Error Estimates for discontinuous Galerkin approximations to the Vlasov-unsteady Stokes system }
	
	\bibliographystyle{alpha}
	
	\author[Harsha Hutridurga]{Harsha Hutridurga}
	\address{H.H.: Department of Mathematics, Indian Institute of Technology Bombay, Powai, Mumbai 400076 India.}
	\email{hutri@math.iitb.ac.in}
	
	\author[Krishan Kumar]{Krishan Kumar}
	\address{K.K.: Department of Mathematics, BITS Pilani, K K Birla Goa Campus, NH 17 B, Zuarinagar, Goa, India.}
	\email{krishan1624@gmail.com}
	
	\author[Amiya K. Pani]{Amiya K. Pani}
	\address{A.K.P. (Corresponding Author): Department of Mathematics, BITS Pilani, K K Birla Goa Campus, NH 17 B, Zuarinagar, Goa, India. }
	\email{amiyap@goa.bits-pilani.ac.in}
	
	\date{\today}
	
	\thispagestyle{empty}
		
		\begin{abstract}
In the first part of this paper, uniqueness of strong solution is established for the Vlasov-unsteady Stokes problem in $3$D. The second part deals with a semi discrete scheme, which is based on the coupling of discontinuous Galerkin approximations for the Vlasov and the Stokes equations for the $2$D problem. The proposed method is both mass and momentum conservative. Based on a special projection and also the Stokes projection, optimal error estimates in the case of smooth compactly supported initial data are derived. Moreover, the generalization of error estimates to $3$D problem is also indicated. Finally, based on time splitting  algorithm, some numerical experiments are conducted whose results confirm our theoretical findings.
		\end{abstract}
		
		\maketitle
	\textbf{Key words.} Vlasov-unsteady Stokes, uniqueness in 3D, discontinuous Galerkin method, conservation properties, optimal error estimates, nonlinear version of  Gr\"onwall's inequality, numerical experiments.

	\textbf{ AMS Subject Classification (MSC 2020).}  35D35, 65N30, 65M15

	\section{Introduction}\label{Int}
	
This paper deals with error estimates for discontinuous Galerkin method for a coupled system of partial differential equations arising in the study of \emph{thin sprays}. From a modeling perspective, it is assumed that the spray particles (droplets) are a dispersed phase in a gas medium. Studying two-phase models comprising of a kinetic equation for the dispersed phase and a fluid equation for the gas dates back to the works of O'Rourke \cite{o1981collective} and Williams \cite{williams2018combustion} (also, see \cite{caflisch1983dynamic}).

We deal with the unsteady Stokes equation for the fluid part and the Vlasov equation for the droplet distribution while the coupling is via a drag term:
\begin{equation}\label{eq:continuous-model}
	\left\{
	\begin{aligned}
		\partial_t f + v\cdot \nabla_x f + \nabla_v \cdot \Big( \left( \bm{u} - v \right) f \Big) & = 0 \qquad \qquad  \mbox{ in }(0,T)\times\Omega_x\times\R^d,
		\\
		f(0,x,v) & = f_{0}(x,v)\quad  \mbox{ in }\Omega_x\times\R^d.
	\end{aligned}
	\right.
\end{equation}
\begin{equation}\label{contstokes}
	\left\{
	\begin{aligned}
		\partial_t \bm{u} - \Delta_x \bm{u} +\nabla_x p & = \int_{\R^d}\left(v - \bm{u}\right)f\,{\rm d}v  \quad\quad \mbox{ in }\,\,\,(0,T)\times\Omega_x,
		\\
		\nabla_x \cdot \bm{u} & = 0 \qquad\qquad\qquad\qquad\,\, \mbox{ in }\,\,\,\Omega_x,
		\\
		\bm{u}(0,x) &= \bm{u_0}(x) \qquad\qquad\qquad\quad \mbox{in}\,\,\,\Omega_x.
	\end{aligned}
	\right.
\end{equation}	
Here, $\Omega_x$ denotes the $d$-dimensional torus $\mathbb{T}^d$ for $d = 2,3$, $\bm{u}(t,x)$ is the fluid velocity, the fluid pressure $p(t,x)$ and the droplet distribution function $f(t,x,v)$. We impose periodic boundary conditions in the $x$ variable. The above model with homogeneous Dirichlet boundary condition for the fluid velocity and with specular reflection boundary condition for the droplet distribution was studied by Hamdache in \cite{Hamdache_1998}, wherein he proved the existence of global-in-time weak solutions. Various kinetic-fluid equations have been studied in the literature: Vlasov-Burgers equations \cite{domelevo1999existence, goudon2001asymptotic}; Vlasov-Euler equations \cite{baranger2006coupling}; Vlasov-Navier-Stokes equations \cite{boudin2009global, chae2011global, yu2013global, han2019uniqueness} and references therein.

In this paper, we provide certain qualitative and quantitative properties of solutions to our system \eqref{eq:continuous-model}-\eqref{contstokes}. While a proof of the global-in-time existence of strong solutions to the system \eqref{eq:continuous-model}-\eqref{contstokes} in $3$D is discussed in \cite{chae2011global}, but uniqueness is missing therefore, our first attempt is to prove uniqueness in $3$D.

The emphasis of this paper is on a numerical method that we are proposing to approximate solutions to the system \eqref{eq:continuous-model}-\eqref{contstokes}. Our current contribution is a first step towards developing a robust numerical scheme for approximating solutions to a more physically relevant model such as the Vlasov-Navier-Stokes equations. The proposed semi-discrete scheme uses discontinuous Galerkin method for both kinetic and unsteady Stokes equations. Our present work is inspired by \cite{ayuso2009discontinuous, heath2012discontinuous, de2012discontinuous} for the Vlasov-Poisson system, \cite{hutridurga2024discontinuous} for the Vlasov-steady Stokes and  \cite{cheng2014discontinuous, yang2017discontinuous} for the Vlasov-Maxwell system. kinetic equation is a transport problem which conserves total mass. As discontinuous Galerkin(dG) methods have the property to preserve the conservation property, dG turns out to be a method of choice for the kinetic equation. It should be noted that either the local discontinuous Galerkin or any conforming numerical method can be employed for the Stokes system. Our major contributions in this paper are as follows:
	\begin{itemize}
	\item Uniqueness of the strong solutions in three dimensions is proved for the Vlasov-unsteady Stokes system.
	Earlier in \cite{hutridurga2023periodic}, there is an error in the proof of uniqueness and a correct proof of uniqueness is given in Section \ref{cts}  of this paper. This can be achieved after proving some higher moment estimates  for the droplet distribution  and regularity results  for the fluid velocity vector.
		\item DG methods are proposed for numerical approximations and are shown to conserve mass and the total momentum, which confirm similar conservation properties for the continuous system.
		\item Optimal error estimates are derived for the fluid velocity and the fluid pressure approximations using Stokes projection in $2$D setup, while in $3$D case some remarks are given in the subsection  \ref{3Dsubsec}.  Moreover, Optimal error estimate for the approximation of the droplet distribution  is obtained using a special projections. The error analysis uses a variant of the nonlinear Gronwall's inequality in a crucial way.
		\item Since phase space is four or six dimensions  depending on $d=2$ or $d=3,$  a Lie splitting in time marching combined with dimension splitting  in the phase space is proposed for  computational experiments, whose  numerical results confirm our theoretically findings. 
	\end{itemize}
	The paper is structured as follows. In Section \ref{cts}, we deal with certain qualitative and quantitative aspects of the solutions to the continuum model. We further prove uniqueness result for strong solutions to \eqref{eq:continuous-model}-\eqref{contstokes}. In Section  \ref{dg}, we introduce a semi-discrete dG-dG numerical scheme and analyse some of its properties. The error analysis for the said semi-discrete method is detailed in Section \ref{err}. We give some concluding remarks in Section \ref{sec:conclude}. Finally, some numerical results confirming our theoretical results are discussed. 
	%
	\section{Certain aspects of the continuum model}\label{cts}

In this paper, let$\rho$ and $V$ denote the local density and the local macroscopic velocity of the droplets, respectively, as
	\[
	\rho(t,x) = \int_{\R^d} f(t,x,v)\,{\rm d}v  \quad \mbox{and} \quad V(t,x) = \frac{1}{\rho}\int_{\R^d} f(t,x,v)v\,{\rm d}v.
	\]	
Define the $l^{th}$ order velocity moments of distribution function as
	\[
	m_lf(t,x) = \int_{\R^d}|v|^lf(t,x,v)\,{\rm d}v, \quad \mbox{for} \quad l \in \N\cup \{0\}.
	\]
Throughout this manuscript,  standard notations for Sobolev spaces with their norms are used, say, the space $W^{k,p}$  denotes the $k^{th}$ order $L^p$-Sobolev spaces with norm $\bm{W}^{k,p} = \left(W^{k,p}(\Omega_x)\right)^2,$  for all $k \geq 0,$ and  for$ 1 \leq p \leq \infty.$  When $p=\infty$,  it is defined in the usual manner.

For our subsequent use,  denote  $\bm{J}_0$ and $\bm{J}_1$ as a special class of divergence-free (in the sense of distributions) vector fields defined by
	\[
	\bm{J}_0 = \left\{\bm{w} \in \bm{L}^2: \nabla_x\cdot\bm{w} = 0, \bm{w} \,\, \mbox{is periodic}\right\},
	\]
	\[
	\bm{J}_1 = \left\{\bm{w} \in \bm{W}^{1,2}: \nabla_x\cdot \bm{w} = 0, \bm{w} \,\, \mbox{is periodic} \right\}.
	\] 
Throughout this paper, any function defined on $\Omega_x$ is assumed to be periodic in the $x$-variable.

\subsection{Qualitative and quantitative aspects of the model problem}

We begin this subsection by stating a result on $L^\infty$ estimate for the local density.  
\begin{lem}\label{lem:rho}
	For $\bm{u} \in L^1(0,T;\bm{L}^\infty),$  if  $\sup_{C^r_{t,v}}f_0 \in L^\infty_{loc}\left(\R_+;L^1(\R^d)\right)$, where $C^r_{t,v} := \Omega_x \times B(e^tv,r), \, \forall\, r > 0$ and $B(e^tv,r)$ denotes the ball centered at $e^tv$ having radius $r$, then, the following estimate holds:
	\begin{equation}\label{rholinf}
		\|\rho\|_{L^{\infty}((0,T]\times\Omega_x)} \leq e^{dT}\sup_{t \in [0,T]}\|\sup_{C^r_{t,v}}f_0\|_{L^1(\R^d)}.
	\end{equation}     	
\end{lem}
\begin{proof}
	The characteristic curves $(X,Y)$ for the equation \eqref{eq:continuous-model} are given by
	\begin{equation}\label{ttts}
		\left\{
		\begin{aligned}
			\dot X(s;t,x,v) &= Y(s;t,x,v); \quad X(t;t,x,v) = x,
			\\
			\dot Y(s;t,x,v) &= \bm{u}(s,X(s;t,x,v)) - Y(s;t,x,v); \quad Y(t;t,x,v) = v,
			\\
			\dot f(s;t,x,v) &= df(s;t,x,v); \quad f(t;t,x,v) = f_0.
		\end{aligned}
		\right.
	\end{equation}
	It is understood that the spatial characteristic curves are periodic. For $0 \leq s \leq t \leq T$, note that
	\begin{equation*}
		\begin{aligned}
			|e^tv - e^sY(s;t,x,v)| &\leq \int_s^te^\tau |\bm{u}(\tau,X(\tau;t,x,v))|\,{\rm d}\tau
			\\
			&\leq e^T\|\bm{u}\|_{L^1(0,T;\bm{L}^\infty)}.
		\end{aligned}
	\end{equation*}
	Therefore, $\left(X(0;t,x,v),Y(0;t,x,v)\right) \in C^r_{t,v}$, with $r = e^T\|\bm{u}\|_{L^1(0,T;\bm{L}^\infty)}$.\\
	From \eqref{ttts}, we obtain 
	\begin{equation*}\label{fexact}
		f(t,x,v) = e^{dt}f_0(X(0;t,x,v),Y(0;t,x,v)).
	\end{equation*}
	Hence,
	\begin{equation*}\label{rhoexact}
		\begin{aligned}
			\rho(t,x) = \int_{\R^d}f(t,x,v)\,{\rm d}v = e^{d t}\int_{\R^d}f_0(X(0;t,x,v),Y(0;t,x,v))\,{\rm d}v
		\end{aligned}
	\end{equation*}
	implying \eqref{rholinf}. This concludes the proof.
\end{proof}
The proof of the above lemma is inspired by \cite[Proposition 4.6, p.44]{han2019uniqueness}. In the following Lemma, we gather certain conservation properties of solutions $(f,\bm{u})$ to the Vlasov-Stokes system \eqref{eq:continuous-model}-\eqref{contstokes}, the proof of which can be found in \cite{Hamdache_1998}. Hence, we skip its proof.
\begin{lem}
	The following properties is satisfied by the solutions $(f,\bm{u})$ to the Vlasov-Stokes system:
	\begin{enumerate}
		\item \textbf{(Positivity preserving)} For any given non-negative initial data $f_0$, the solution $f$ remains non-negative.
		\item \textbf{(Mass conservation)} The total mass of the solution $f$ is conserved in the following sense:
		\begin{equation*}
			\int_{\R^d}\int_{\Omega_x}\,f(t,x,v)\,{\rm d}x\,{\rm d}v = \int_{\R^d}\int_{\Omega_x}\,f_0(x,v)\,{\rm d}x\,{\rm d}v, \quad t \in [0,T].
		\end{equation*}
	\item \textbf{(Total momentum conservation)} The total momentum of the solution pair $\left(f,\bm{u}\right)$ is preserved in the following sense: for all $t \in [0,T]$,
	\begin{equation*}
		\int_{\R^d}\int_{\Omega_x} vf(t,x,v)\,{\rm d}x\,{\rm d}v + \int_{\Omega_x} \bm{u}(t,x)\,{\rm d}x = \int_{\R^d}\int_{\Omega_x} vf_0(x,v)\,{\rm d}x\,{\rm d}v + \int_{\Omega_x}\bm{u_0}(x)\,{\rm d}x.
	\end{equation*}
	\end{enumerate}
\end{lem}

The following Lemma gives the important energy-dissipation identity satisfied by the solutions to the Vlasov-Stokes system.
\begin{lem}\textbf{(Energy dissipation)}\label{lem:energy-dissipate}
	The total energy of the Vlasov-Stokes system \eqref{eq:continuous-model}-\eqref{contstokes} dissipates in time, i.e.,
	\begin{equation*}
		\frac{{\rm d}}{{\rm d}t}\left(\int_{\R^d}\int_{\Omega_x}|v|^2f(t,x,v)\,{\rm d}x\,{\rm d}v + \int_{\Omega_x}\bm{u}^2\,{\rm d}x \right) \leq 0,
	\end{equation*}
	provided $f$ is non-negative.
\end{lem}
\begin{proof}
	Multiplying equation \eqref{eq:continuous-model} by $\frac{|v|^2}{2}$, and equation \eqref{contstokes} by $\bm{u} \in \bm{J}_1$ followed by integration by parts, we obtain
	\begin{equation*}
		\begin{aligned}
			\int_{\R^d}\int_{\Omega_x} \partial_t\left(\frac{|v|^2}{2}f\right)\,{\rm d}x\,{\rm d}v &+ \frac{1}{2}\int_{\Omega_x}\partial_t\bm{u}^2\,{\rm d}x + \int_{\Omega_x}|\nabla_x\bm{u}|^2\,{\rm d}x - 2\int_{\R^d}\int_{\Omega_x} v\cdot\bm{u}f\,{\rm d}x\,{\rm d}v
			\\
			& + \int_{\R^d}\int_{\Omega_x}|v|^2f\,{\rm d}x\,{\rm d}v + \int_{\R^d}\int_{\Omega_x}|\bm{u}|^2f\,{\rm d}x\,{\rm d}v = 0.
		\end{aligned}
	\end{equation*}
	The above equality is the same as
	\begin{equation*}
		\frac{{\rm d}}{{\rm d}t}\left( \int_{\R^d}\int_{\Omega_x} |v|^2\;f \,{\rm d}x\,{\rm d}v + \int_{\Omega_x}\bm{u}^2 {\rm d}x \right) + 2 \int_{\Omega_x}|\nabla_x\bm{u}|^2\,{\rm d}x + 2 \int_{\R^d}\int_{\Omega_x}|\bm{u} - v|^2f\,{\rm  d}x\,{\rm d}v = 0.
	\end{equation*}
	As the third term is non-negative and as the last term is non-negative because $f$ is non-negative, we drop them to obtain the desired result. This completes the proof.
\end{proof}
As a consequence of above lemma, we also obtain the following identity:
\begin{equation}\label{ene}
	\begin{aligned}
		\frac{1}{2}&\left(\int_{\R^d}\int_{\Omega_x} |v|^2f(t,x,v)\,{\rm d}x\,{\rm d}v  + \int_{\Omega_x}\bm{u}^2\,{\rm d}x\right) + \int_0^t\int_{\Omega_x}|\nabla_{x}\bm{u}|^2\,{\rm d}x\,{\rm d}t 
		\\
		&+ \int_0^t\int_{\R^d}\int_{\Omega_x}|\bm{u} - v|^2f\,{\rm d}x\,{\rm d}v\,{\rm d}t = \frac{1}{2}\int_{\R^d}\int_{\Omega_x} |v|^2\,f_0\,{\rm d}x\,{\rm d}v + \frac{1}{2}\int_{\Omega_x}\bm{u_0}^2\,{\rm d}x.
	\end{aligned}  
\end{equation}
Hence, we deduce that
\begin{equation}\label{uL2}
	\bm{u} \in L^\infty(0,T;\bm{L}^2) \quad \mbox{and} \quad \bm{u} \in L^2(0,T;\bm{J}_1)
\end{equation}
provided $|v|^2f_0 \in L^1(\Omega_x \times \R^d)$ and $\bm{u_0} \in \bm{L}^2$.\\
Recall by Sobolev imbedding $H^1(\Omega_x) \hookrightarrow   L^p(\Omega_x)$. Here, when $d = 2, ~  p \in [2,\infty)$ and for $d = 3,~  p \in [2,6]$:
\begin{equation}\label{uLp}
	\bm{u} \in L^2(0,T;\bm{L}^p).
\end{equation}
Below, we state a result on the integrability of the local density and local momentum for which proof can be found in \cite[Lemma 2.2, p.56]{Hamdache_1998}. These results are essential to obtain the regularity of the solutions to the Stokes equations.
\begin{lem}\label{density}
	For $p \geq 1$, let $\bm{u} \in L^2(0,T;\bm{L}^{p+d}), f_0 \in L^\infty(\Omega_x \times \R^d)\cap L^1(\Omega_x \times \R^d)$ and 
	\[
	\int_{\R^d}\int_{\Omega_x}\,|v|^pf_0\,{\rm d}x{\rm d}v <\infty.
	\] 
 Then, the local density $\rho$ and the local momentum $\rho V$ satisfy the following:
	\begin{equation}\label{rhos}
		\rho \in L^\infty\left(0,T;L^\frac{p+d}{d}(\Omega_x)\right) \quad \mbox{and} \quad \rho V \in L^\infty\left(0,T; L^\frac{p+d}{d+1}(\Omega_x)\right).
	\end{equation}
\end{lem}
\begin{rem}
   	  By choosing $p = 5$ in Lemma \ref{density} for $d = 2$, we obtain 
   	  \begin{equation}\label{rhos1}
   	  	\rho \in L^\infty\left(0,T;L^\frac{7}{2}(\Omega_x)\right) \quad \mbox{and} \quad \rho V \in L^\infty\left(0,T; L^\frac{7}{3}(\Omega_x)\right).
   	  \end{equation}
   	\end{rem}
   	\begin{rem}
   		 For $d = 3$, $p = 3$ in Lemma \ref{density} yields 
   		\begin{equation}\label{rhos1}
   			\rho \in L^\infty\left(0,T;L^2(\Omega_x)\right) \quad \mbox{and} \quad \rho V \in L^\infty\left(0,T; L^\frac{3}{2}(\Omega_x)\right).
   		\end{equation}
   		An application of the Stokes regularity results \cite{giga1991abstract,su2023hydrodynamic} yields
   		\begin{equation}
   			\bm{u} \in L^2(0,T;\bm{W}^{2,\frac{3}{2}}).
   		\end{equation}
   		An use of the Sobolev inequality shows
   		\begin{equation}
   			\bm{u} \in L^2(0,T;\bm{L}^p) \quad \mbox{for}\quad \frac{3}{2} \leq p < \infty.
   		\end{equation}
   			 Taking $p = 5$ in Lemma \ref{density}, we obtain 
   			\begin{equation}\label{rhos1}
   				\rho \in L^\infty\left(0,T;L^\frac{8}{3}(\Omega_x)\right) \quad \mbox{and} \quad \rho V \in L^\infty\left(0,T; L^2(\Omega_x)\right).
   			\end{equation}
   			A use of the Stokes regularity result yields
   			\begin{equation}
   				\bm{u} \in H^1(0,T;\bm{L}^2)\cap L^2(0,T;\bm{H}^2)\cap L^\infty(0,T;\bm{H}^1).
   			\end{equation}
   			Setting $p = 9+\delta$ with $\delta > 0$ in Lemma \ref{density} shows 
   			\begin{equation}\label{rhos1}
   				\rho \in L^\infty\left(0,T;L^\frac{12+\delta}{3}(\Omega_x)\right) \quad \mbox{and} \quad \rho V \in L^\infty\left(0,T; L^\frac{12+\delta}{4}(\Omega_x)\right).
   			\end{equation}
   		\end{rem}
    These integrability properties of $\rho$ and $\rho V$ in the Stokes equation \eqref{contstokes} guarantees that $\bm{u} \in L^2(0,T;\bm{W}^{1,\infty})$ which is important for our proof of existence of a unique strong solution.
   
Below, we state and prove a result on the propagation of velocity moments.
\begin{lem}\label{mkmmm}
  	Let $\bm{u} \in L^1(0,T;\bm{W}^{1,\infty})$ and let $f_0\ge0$ be such that
  	\[
  	\int_{\Omega_x}\int_{\R^3}\langle v\rangle^{kp}\{ \vert f_0\vert^p + |\nabla_xf_0|^p + |\nabla_vf_0|^p\}\,{\rm d}v\,{\rm d}x \leq C,
  	\]
  	for $k \geq 0$ and $\langle v\rangle = \left(1 + \vert v\vert^2\right)^\frac{1}{2}$. Then, for $k \geq 0$ the solution $f$ of the kinetic equation satisfies
  	\[
  	\int_{\Omega_x}\int_{\R^3}\langle v\rangle^{kp}\{\vert f\vert^p + |\nabla_xf|^p + |\nabla_vf|^p\}\,{\rm d}v\,{\rm d}x \leq C \quad \forall ~ t > 0.
  	\]
	  
\end{lem}
   
   \begin{proof}
   	Consider the equation for $\frac{\partial f}{\partial x_i}$:
   	\begin{equation*}
   		\partial_t\frac{\partial f}{\partial x_i} + v\cdot \nabla_x\frac{\partial f}{\partial x_i} + \nabla_v\cdot \left(\frac{\partial {\bm u}}{\partial x_i} f\right) + \nabla_v\cdot\left(\left(\bm{u} - v\right)\frac{\partial f}{\partial x_i}\right) = 0,
   	\end{equation*}
   	for $i=1,2,3$. Multiplying the above vector equation by $\langle v\rangle^{kp}\vert\nabla_xf\vert^{p-1}$, an integration with respect to $x,v$ yields
   	\begin{equation*}
   		\begin{aligned}
   			\frac{1}{p}\frac{{\rm d}}{{\rm d}t}\int_{\Omega_x}\int_{\R^3}\langle v\rangle^{kp}|\nabla_xf|^p\,{\rm d}v\,{\rm d}x = I_1 + I_2 + I_3,
   		\end{aligned}
   	\end{equation*} 
   	where
   	\begin{align*}
   		I_1 & = -\int_{\Omega_x}\int_{\R^3}\langle v\rangle^{kp}\vert\nabla_xf\vert^{p-1}\nabla_x\bm{u}\cdot \nabla_vf\,{\rm d}v\,{\rm d}x,
   		\\
   		I_2 & =  3\int_{\Omega_x}\int_{\R^3}\langle v\rangle^{kp}|\nabla_xf|^p\,{\rm d}v\,{\rm d}x,
   		\\
   		I_3 & = -\frac{1}{p}\int_{\Omega_x}\int_{\R^3}\langle v\rangle^{kp}\left(\bm{u} - v\right)\cdot\nabla_v\left(|\nabla_xf|^p\right)\,{\rm d}v\,{\rm d}x.
   	\end{align*}
   	After using Young's inequality in $I_1$, we obtain
   	\begin{equation*}
   		\begin{aligned}
   			I_1 \le \|\nabla_x\bm{u}\|_{\bm{L}^\infty}\int_{\Omega_x}\int_{\R^3}\langle v\rangle^{kp}\left(\frac{p-1}{p}|\nabla_xf|^p + \frac{1}{p}|\nabla_vf|^p\right)\,{\rm d}v\,{\rm d}x.
   		\end{aligned}
   	\end{equation*}
   	An integration by parts yields
   	\begin{equation*}
   		\begin{aligned}
   			I_3 = -\frac{3}{p}\int_{\Omega_x}\int_{\R^3}\langle v\rangle^{kp}|\nabla_xf|^p\,{\rm d}v\,{\rm d}x + I_4
   		\end{aligned}
   	\end{equation*}
   	with
   	\begin{equation*}
   		I_4 =  k\int_{\Omega_x}\int_{\R^3}\,\langle v\rangle^{kp-2}v\cdot\left(\bm{u} - v\right)|\nabla_xf|^p\,{\rm d}v\,{\rm d}x.
   	\end{equation*}
   	A use of the Young's inequality shows
   	\begin{equation*}
   		\begin{aligned}
   			I_4 &\leq k\|\bm{u}\|_{\bm{L}^\infty}\int_{\Omega_x}\int_{\R^3}\,\langle v\rangle^{kp-1}|\nabla_xf|^p\,{\rm d}v\,{\rm d}x + k\int_{\Omega_x}\int_{\R^3}\,\langle v\rangle^{kp}|\nabla_xf|^p\,{\rm d}v\,{\rm d}x
   			\\
   			&\leq C\left(1 + \|\bm{u}\|_{\bm{L}^\infty}\right)\int_{\Omega_x}\int_{\R^3}\langle v\rangle^{kp}|\nabla_xf|^p\,{\rm d}v\,{\rm d}x.
   		\end{aligned}
   	\end{equation*}
   	A similar computation involving the equation for $\nabla_v f$ yields
   	\begin{equation*}
   		\begin{aligned}
   			\frac{1}{p}\frac{{\rm d}}{{\rm d}t}\int_{\Omega_x}\int_{\R^3}\langle v\rangle^{kp}|\nabla_vf|^p\,{\rm d}v\,{\rm d}x & \le C\int_{\Omega_x}\int_{\R^3}\langle v\rangle^{kp}\left(|\nabla_xf|^p + |\nabla_vf|^p\right)\,{\rm d}v\,{\rm d}x
   			\\
   			& + C\left(1 + \|\bm{u}\|_{\bm{L}^\infty}\right)\int_{\Omega_x}\int_{\R^3}\langle v\rangle^{kp}|\nabla_vf|^p\,{\rm d}v\,{\rm d}x. 
   		\end{aligned}
   	\end{equation*}
   	Altogether, we obtain
   	\begin{equation*}
   		\begin{aligned}
   			\frac{{\rm d}}{{\rm d}t}&\left(\int_{\Omega_x}\int_{\R^3}\langle v\rangle^{kp}\left(|\nabla_xf|^p + |\nabla_vf|^p\right)\,{\rm d}v\,{\rm d}x\right) \leq C\left(1 + \|\bm{u}\|_{\bm{W}^{1,\infty}}\right)
   			\\
   			&\qquad\qquad\qquad\qquad\qquad\qquad\qquad\int_{\Omega_x}\int_{\R^3}\langle v\rangle^{kp}\left(|\nabla_xf|^p + |\nabla_vf|^p\right)\,{\rm d}v\,{\rm d}x.
   		\end{aligned}
   	\end{equation*}
   	A use of Gr\"onwall's inequality  completes the rest of the proof.
   \end{proof}

\subsection{Existence and Uniqueness result in $3$D}
This subsection discusses briefly the existence and uniqueness result for the strong solution to \eqref{eq:continuous-model}-\eqref{contstokes} the continuum model.

Below, we stated the result on the existence of strong solution to the Vlasov-Stokes equation whose proof can be found in \cite[Theorem 5, p. 2435]{chae2011global}.
\begin{thm}\textbf{(Existence of strong solution)}\label{ext:strong}
	Let the initial data $(f_0,\bm{u_0})$ be such that
	\[
	f_0 \in L^1(\Omega_x\times\R^3)\cap L^\infty(\Omega_x\times\R^3), \quad f_0\ge0,
	\]
	\[
	\int_{\Omega_x}\int_{\R^3}\langle v\rangle^{kp}\left\{ \vert f_0\vert^p + \vert \nabla_xf_0\vert^p + \vert \nabla_vf_0\vert^p \right\}\,{\rm d}v\,{\rm d}x \leq C,
	\]
	for $p \in (3,\infty), \quad k > 4 - \frac{3}{p}$ and $\bm{u_0}\in \bm{W}^{1,p}\cap\bm{J}_1$. 
	Then, there exists a global-in-time strong solution $(f,\bm{u},p)$ to the Vlasov-Stokes system \eqref{eq:continuous-model}-\eqref{contstokes}. Furthermore, 
	\[
	\int_{\Omega_x}\int_{\R^3}\langle v\rangle^{kp}\left\{ \vert f\vert^p + \vert \nabla_xf\vert^p + \vert \nabla_vf\vert^p \right\}\,{\rm d}v\,{\rm d}x \leq C,
	\]
	and 
	\[
	\bm{u} \in \bm{L}^r(0,T;\bm{W}^{2,q})\cap H^1(0,T;\bm{L}^q)
	\]
	for $q < p$ and $r \in (1,\infty)$. Here, $\langle  v\rangle = \left(1 + \vert v\vert^2\right)^\frac{1}{2}$.
\end{thm}
The following theorem is on uniqueness of the strong solution to \eqref{eq:continuous-model}-\eqref{contstokes}. 
\begin{thm}\textbf{(Uniqueness of strong solution)}\label{thm:exist-strong}
	Under the hypothesis of Theorem \ref{ext:strong} there is atmost one strong solution $(f,\bm{u},p)$ to \eqref{eq:continuous-model}-\eqref{contstokes}.	
\end{thm}
\begin{proof}
	Suppose the solution is not unique, that is, $(f_1,\bm{u}_1,p_1)$ and $(f_2,\bm{u}_2,p_2)$ are two distinct strong solutions of \eqref{eq:continuous-model}-\eqref{contstokes} with say $\bm{u}_1\neq\bm{u}_2, p_1 \neq p_2$ and $f_1 \neq f_2$ . Let $\bar{\bm{u}} = \bm{u}_1 - \bm{u}_2, \bar{p} = p_1 - p_2$ and $\bar{f} = f_1 - f_2$, then $\bar{\bm{u}}, \bar{p}$ and $\bar{f}$ satisfies the following equations:
	\begin{equation}\label{131-1}
		\bar{f}_t + v\cdot\nabla_x\bar{f} + \nabla_v\cdot\left(\bar{\bm{u}}f_1 + \bm{u}_2\bar{f} - v\bar{f}\right) = 0,
	\end{equation}
	and
	\begin{equation}\label{141-1}
		\left\{
		\begin{aligned}
			\bar{\bm{u}}_t - \Delta_x\bar{\bm{u}} + \nabla_xp &= \int_{\R^3}\left(v\bar{f} - \bm{u}_2\bar{f} - \bar{\bm{u}}f_1\right)\,{\rm d}v,
			\\
			\nabla_x\cdot\bar{\bm{u}} &= 0
		\end{aligned}
		\right.
	\end{equation}
	with initial data
	\begin{equation*}
		\bar{f}(0,x,v) = 0, \qquad \bar{\bm{u}}(0,x) = 0.
	\end{equation*}
	By multiplying the equation \eqref{141-1} by $\bar{\bm{u}}$ and integrating in $x$, we obtain
	\begin{equation}\label{eq:a17}
		\begin{aligned}
			\frac{1}{2}\frac{{\rm d}}{{\rm d}t}\|\bar{\bm{u}}\|^2_{\bm{L}^2} + \|\nabla_x\bar{\bm{u}}\|_{\bm{L}^2} &= \int_{\Omega_x}\bar{\bm{u}}\left(\int_{\R^3}v\bar{f}\,{\rm d}v\right){\rm d}x - \int_{\Omega_x}\vert\bar{\bm{u}}\vert^2\left(\int_{\R^3}f_1\,{\rm d}v\right){\rm d}x 
			\\
			& \qquad - \int_{\Omega_x} \bm{u}_2\bar{\bm{u}}\left(\int_{\R^3}\bar{f}\,{\rm d}v\right){\rm d}x 
			\\
			&\leq \left(\int_{\Omega_x}\vert\bar{\bm{u}}\vert^2\left(\int_{\R^3}\frac{1}{\langle v\rangle^{2k-2}}\,{\rm d}v\right){\rm d}x\right)^\frac{1}{2}\left(\int_{\Omega_x}\int_{\R^3}\langle v\rangle^{2k}\vert\bar{f}\vert^2\,{\rm d}v\,{\rm d}x\right)^\frac{1}{2}
			\\
			& + \|\bm{u}_2\|_{\bm{L}^\infty}\left(\vert\bar{\bm{u}}\vert^2\left(\int_{\R^3}\frac{1}{\langle v\rangle^{2k}}\,{\rm d}v\right){\rm d}x\right)^\frac{1}{2}\left(\int_{\Omega_x}\int_{\R^3}\langle v\rangle^{2k}\vert \bar{f}\vert^2\,{\rm d}v\,{\rm d}x\right)^\frac{1}{2}
			\\
			& \leq C\|\bar{\bm{u}}\|^2_{\bm{L}^2} + \|\langle v\rangle^k\bar{f}\|^2_{L^2(\Omega_x\times\R^3)},
		\end{aligned}
	\end{equation}
	here, in the second step use the H\"older inequality and in the last step use the Young's inequality with assumption that $k > \frac{5}{2}$.
	
	Now, multiply equation \eqref{131-1} by $\langle v \rangle^{2k}\bar{f}$ and integrate in $x,v$-variables, to obtain
	\begin{equation}
		\begin{aligned}
			\frac{1}{2}\frac{{\rm d}}{{\rm d}t}&\|\langle v\rangle^k\bar{f}\|^2_{L^2(\Omega_x\times\R^2)} + \frac{1}{2}\int_{\Omega_x}\int_{\R^3}\langle v\rangle^{2k}v\cdot\nabla_x\bar{f}^2\,{\rm d}v\,{\rm d}x 
			\\
			&= -\int_{\Omega_x}\int_{\R^3}\langle v\rangle^{2k}\bar{f}\,\bar{\bm{u}}\cdot\nabla_vf_1\,{\rm d}v\,{\rm d}x 
			 - \frac{1}{2}\int_{\Omega_x}\int_{\R^3}\langle v\rangle^{2k}\bar{\bm{u}}_2\cdot\nabla_v\bar{f}^2\,{\rm d}v\,{\rm d}x 
			 \\
			 &\qquad + 3\int_{\Omega_x}\int_{\R^3}\langle v\rangle^{2k}\bar{f}^2\,{\rm d}v\,{\rm d}x + \frac{1}{2}\int_{\Omega_x}\int_{\R^3}\langle v\rangle^{2k}v\cdot\nabla_v\bar{f}^2\,{\rm d}v\,{\rm d}x.
		\end{aligned}
	\end{equation}
	A use of integration by parts yields
	\begin{equation}\label{eq:a19}
		\begin{aligned}
			\frac{1}{2}\frac{{\rm d}}{{\rm d}t}\|\langle v\rangle^k\bar{f}\|^2_{L^2(\Omega_x\times\R^3)} = \tilde{T}_1 + \tilde{T}_2 + \tilde{T}_3
		\end{aligned}
	\end{equation}
	where
	\begin{equation*}
		\tilde{T}_1 = -\int_{\Omega_x}\int_{\R^3}\langle v\rangle^{2k}\bar{f}\,\bar{\bm{u}}\cdot\nabla_vf_1\,{\rm d}v\,{\rm d}x 
	\end{equation*}
	\begin{equation*}
		\tilde{T}_2 = -\frac{1}{2}\int_{\Omega_x}\int_{\R^3}\langle v\rangle^{2k}\bar{\bm{u}}_2\cdot\nabla_v\bar{f}^2\,{\rm d}v\,{\rm d}x
	\end{equation*}
	\begin{equation*}
		\tilde{T}_3 =  3\int_{\Omega_x}\int_{\R^3}\langle v\rangle^{2k}\bar{f}^2\,{\rm d}v\,{\rm d}x + \frac{1}{2}\int_{\Omega_x}\int_{\R^3}\langle v\rangle^{2k}v\cdot\nabla_v\bar{f}^2\,{\rm d}v\,{\rm d}x.
	\end{equation*}
 	For $\tilde{T}_1$ term
	\begin{equation}\label{eq:a20}
		\begin{aligned}
			\tilde{T}_1 &\leq \left(\int_{\Omega_x}\int_{\R^3}\langle v\rangle^{2k}\bar{f}^2\,{\rm d}v\,{\rm d}x\right)^\frac{1}{2}\left(\int_{\Omega_x}\vert\bar{\bm{u}}\vert^5\left(\int_{\R^3}\frac{1}{\langle v\rangle^{5\alpha}}\,{\rm d}v\right)\,{\rm d}x\right)^\frac{1}{5}\cdot
			\\
			&\quad\left(\int_{\Omega_x}\int_{\R^3}\langle v\rangle^\frac{10(k+\alpha)}{3}(\nabla_vf_1)^\frac{10}{3}\,{\rm d}v\,{\rm d}x\right)^\frac{3}{10}
			\\
			&\leq \|\langle v\rangle^k\bar{f}\|_{L^2(\Omega_x\times\R^3)}\|\bar{\bm{u}}\|_{\bm{L}^5}\|\langle v\rangle^{k+\alpha}\nabla_vf_1\|_{L^\frac{10}{3}(\Omega_x\times\R^3)}
			\\
			&\leq \|\langle v\rangle^k\bar{f}\|_{L^2(\Omega_x\times\R^3)}\left(\|\bar{\bm{u}}\|_{\bm{L}^2}^\frac{1}{10}\|\nabla\bar{\bm{u}}\|_{\bm{L}^2}^\frac{9}{10} + \|\bar{\bm{u}}\|_{\bm{L}^2}\right)\|\langle v\rangle^{k+\alpha}\nabla_vf_1\|_{L^\frac{10}{3}(\Omega_x\times\R^3)}
			\\
			&\leq \frac{1}{2}\|\langle v\rangle^k\bar{f}\|^2_{L^2(\Omega_x\times\R^3)}\|\langle v\rangle^{k+\alpha}\nabla_vf_1\|^2_{L^\frac{10}{3}(\Omega_x\times\R^3)} + \frac{1}{20}\|\bar{\bm{u}}\|^2_{\bm{L}^2} + \frac{9}{20}\|\nabla\bar{\bm{u}}\|^2_{\bm{L}^2} 
			\\
			&\qquad + \frac{1}{2}\|\langle v\rangle^k\bar{f}\|^2_{L^2(\Omega_x\times\R^3)}\|\langle v\rangle^{k+\alpha}\nabla_vf_1\|^2_{L^\frac{10}{3}(\Omega_x\times\R^3)} + \frac{1}{2}\|\bar{\bm{u}}\|^2_{\bm{L}^2}
			\\
			&\leq \|\langle v\rangle^{k+\alpha}\nabla_vf_1\|^2_{L^\frac{10}{3}(\Omega_x\times\R^3)}\|\langle v\rangle^k\bar{f}\|^2_{L^2(\Omega_x\times\R^3)} + 2\|\bar{\bm{u}}\|^2_{\bm{L}^2} + \frac{9}{20}\|\nabla\bar{\bm{u}}\|^2_{\bm{L}^2}.
		\end{aligned}
	\end{equation}
	Here, in the second step we have used the H\"older inequality with the assumption that $\alpha > \frac{3}{5}$, in the fourth step the Gagliardo - Nirenberg inequality \cite[Theorem 9.9, p. 1317]{han2020large} and in the fifth step, we have applied the Young's inequality.\\
	A use of integration by parts with the H\"older inequality yields
	\begin{equation}\label{eq:a21}
		\begin{aligned}
			\tilde{T}_2 &= k\int_{\Omega_x}\int_{\R^3}\langle v\rangle^{2k-2}v\cdot\bm{u}_2\bar{f}^2\,{\rm d}v\,{\rm d}x
			\\
			&\leq k\|\bm{u}_2\|_{\bm{L}^\infty}\|\langle v\rangle^k\bar{f}\|^2_{L^2(\Omega_x\times\R^3)}.
		\end{aligned}
	\end{equation}
	Apply the integration by parts, then the $\tilde{T}_3$ term can be estimated as
	\begin{equation}\label{eq:a22}
		\begin{aligned}
			\tilde{T}_3 &= 3\|\langle v\rangle^k\bar{f}\|^2_{L^2(\Omega_x\times\R^3)} - k\int_{\Omega_x}\int_{\R^3}\langle v\rangle^{2k}\bar{f}^2\,{\rm d}v\,{\rm d}x - \frac{3}{2}\int_{\Omega_x}\int_{\R^3} \langle v\rangle^{2k}\bar{f}^2\,{\rm d}v\,{\rm d}x
			\\
			&= -\left(k - \frac{3}{2}\right)\|\langle v\rangle^k\bar{f}\|^2_{L^2(\Omega_x\times\R^3)}. 
		\end{aligned}
	\end{equation}
	Putting the estimates \eqref{eq:a20}-\eqref{eq:a22} into \eqref{eq:a19} shows
	\begin{equation}\label{eq:a23}
		\begin{aligned}
			\frac{1}{2}\frac{{\rm d}}{{\rm d}t}\|\langle v\rangle^k\bar{f}\|^2_{L^2(\Omega_x\times\R^3)} \leq C\|\langle v\rangle^k\bar{f}\|^2_{L^2(\Omega_x\times\R^3)} + 2\|\bar{\bm{u}}\|^2_{\bm{L}^2} + \frac{9}{20}\|\nabla\bar{\bm{u}}\|^2_{\bm{L}^2}. 
		\end{aligned}
	\end{equation}
	Adding equations \eqref{eq:a17} and \eqref{eq:a23}, we obtain
	\begin{equation}
		\begin{aligned}
			\frac{1}{2}\frac{{\rm d}}{{\rm d}t}\left( \|\bar{\bm{u}}\|^2_{\bm{L}^2} + \|\langle v\rangle^k\bar{f}\|^2_{L^2(\Omega_x\times\R^3)}\right) + \frac{11}{20}\|\nabla\bar{\bm{u}}\|^2_{\bm{L}^2} \leq C\left(\|\langle v\rangle^k\bar{f}\|^2_{L^2(\Omega_x\times\R^3)} + \|\bar{\bm{u}}\|^2_{\bm{L}^2}\right). 
		\end{aligned}
	\end{equation}
	A use of the Gr\"onwall's inequality yields
	\begin{equation}
		\|\bar{\bm{u}}\|^2_{\bm{L}^2} + \|\langle v\rangle^k\bar{f}\|^2_{L^2(\Omega_x\times\R^3)} \leq 0.
	\end{equation}
	This leads to a contradiction and $\bm{u}_1 = \bm{u}_2,  f_1 = f_2$. It completes the rest of the proof.
\end{proof}

\section{Discontinuous Galerkin approximations}\label{dg}
This section deals with the discontinuous Galerkin method and certain properties of the discrete system for two dimensional problem \eqref{eq:continuous-model}-\eqref{contstokes}. However, a remark on the three dimensional problem is given in subsection \ref{3Dsubsec}. It is clear that a compactly supported (in the velocity variable $v$) initial datum $f_0$ will yield a compactly supported solution $f(t,x,v)$ in the velocity variable $v$. Restricting ourselves to compactly supported initial data, we can assume without loss of generality that there is some $L>0$ such that for $v \in [-L,L]^2 = \Omega_v$ and $t \in (0,T]; \, \mbox{supp}f(t,x,v) \subset \Omega = \Omega_x \times \Omega_v$. 

Let $\mathcal{T}^x_{h}$ and $\mathcal{T}^v_{h}$ be the Cartesian partitions of $\Omega_x$ and $\Omega_v$, respectively which are shape regular and quasi-uniform. Let $\mathcal{T}_{h}$ be defined as cartesian product of these two partitions, i.e.,
\begin{align*}
	\mathcal{T}_{h} = \left\{ R = T^x \times T^v : T^x \in \mathcal{T}^x_{h}, \,  T^v \in \mathcal{T}^v_{h} \right\}.
\end{align*}
The mesh sizes $h$, $h_x$ and $h_v$ relative to these partitions are defined as follows:
\[
h_x := \max_{T^x \in \mathcal{T}^x_{h}} {\rm diam} (T^x);
\quad
h_v := \max_{T^v \in \mathcal{T}^v_{h}} {\rm diam} (T^v);
\quad
h := \max \left( h_x, h_v \right).
\]
We use $\Gamma_x$ and $\Gamma_v$ to denote the set of all edges of the partitions $\mathcal{T}^x_{h}$ and $\mathcal{T}^v_{h}$, respectively and  $\Gamma = \Gamma_x \times \Gamma_v.$
Further, let $\Gamma^0_{x}$ (respectively, $\Gamma^0_{v}$) and $\Gamma^\partial_{x}$ (respectively, $\Gamma^\partial _{v}$) denote the sets of interior and boundary edges of partition $\mathcal{T}^x_{h}$ (respectively, $\mathcal{T}^v_{h}$), so that $\Gamma_x = \Gamma^0_{x} \cup \Gamma^\partial _{x}$ (respectively, $\Gamma_v = \Gamma^0_{v} \cup \Gamma^\partial _{v}$).

We define the broken polynomial spaces as
\begin{equation*}
	\begin{aligned}
		\bm{H}_h &:= \left\{ \bm{\phi} \in L^2(\Omega_x) : \bm{\phi}\vert_{T^x} \in \left(\mathbb{P}^k(T^x)\right)^2,\, \,\forall\, T^x \in \mathcal{T}_{h}^x \right\},
		\\
		L_h&:= \left\{ \phi \in L^2_0(\Omega_x) : \phi\vert_{T^x} \in \mathbb{P}^k(T^x),\, \,\forall\, T^x \in \mathcal{T}_{h}^x \right\},
		\\
		X_h &:= \{\phi \in L^2(\Omega_x) : \phi\vert_{T^x} \in \mathbb{P}^k(T^x),\, \,\forall\, T^x \in \mathcal{T}_{h}^x\},
		\\
	V_h &:= \{\phi \in L^2(\Omega_v) : \phi\vert_{T^v} \in \mathbb{P}^k(T^v),\, \,\forall\, T^v \in \mathcal{T}_{h}^v\},
\\
	\mathcal{Z}_h &:= \left\{ \psi \in L^2(\Omega_x\times\Omega_v) : \psi\vert_{R} \in \mathbb{P}^k(T^x) \times \mathbb{P}^k(T^{v}),\,\,  \forall R = T^x \times T^v \in \mathcal{T}_{h} \right\},
\end{aligned}
\end{equation*}
where $\mathbb{P}^k(T)$ is the space of scalar polynomials of degree at most $k$ in each variable.

Below, we define the jump and average value on the mesh. Let the inward and outward unit normal vectors on the element $T^r, r = x \,\, \mbox{or}\,\,v$ denoted by $\bm{n}^-$ and $\bm{n}^+$, respectively. Following \cite{arnold2002unified, vemaganti2007discontinuous}, we define the average and jump of a scalar function $\phi$ and a vector-valued function $\bm{\phi}$ at the edges as follows:
\[ 
\{\phi \} = \frac{1}{2}(\phi^- + \phi^+), \hspace{5mm} [\![\phi]\!] = \phi^- \bm{n}^- + \phi^+ \bm{n}^+ \hspace{3mm} \text{on} \hspace{1mm}  \Gamma^0_r, \hspace{3mm} r = x \hspace{1mm}\text{or}\hspace{1mm} v
\]
\[ 
\{\bm{\phi} \} = \frac{1}{2}(\bm{\phi}^- + \bm{\phi}^+), \hspace{5mm} [\![\bm{\phi}]\!] = \bm{\phi}^-\cdot \bm{n}^- + \bm{\phi}^+\cdot \bm{n}^+ \hspace{3mm} \text{on} \hspace{1mm}  \Gamma^0_r, \hspace{3mm} r = x \hspace{1mm}\text{or}\hspace{1mm} v,
\]
where,
\[
\phi^{\pm}_{T^x}(x,.) = \lim_{\epsilon \to 0}\phi_{T^x}(x\pm \epsilon \bm{n}^-,.) \hspace{5mm}\forall \hspace{1mm} x \in \partial T^x.
\]
For a vector valued function $\bm{\phi}$ the weighted average is defined by
\[
\{\bm{\phi}\}_\delta := \delta\bm{\phi}^+ + \left(1 - \delta\right)\bm{\phi}^- \quad \mbox{for}\quad 0\leq \delta\leq 1.
\]
Note that for a fixed edge $\bm{n}^- = -\bm{n}^+$. For the boundary edges, the jump and the average are taken to be $[\![\phi]\!] = \phi\bm{n}$ and $\{\phi\} = \phi$.

The discrete spaces $W^{k,p}(\mathcal{T}_h)$ are defined as
\[
W^{k,p}(\mathcal{T}_h) = \{\phi \in L^p(\Omega): \phi\mid_R \in W^{k,p}(R), ~~ \forall~~ R \in \mathcal{T}_h\} \quad k \geq 0, ~1\leq p \leq \infty.
\]
We further use $H^k(\mathcal{T}_h)$ to denote $W^{k,2}(\mathcal{T}_h)$ for $k \geq 1$. The discrete norms are defined by 
\[
\|z\|_{m,\mathcal{T}_h} = \left(\sum_{R \in \mathcal{T}_h}\|z\|_{m,R}\right)^\frac{1}{2}, ~ \forall~ z \in H^m(\mathcal{T}_h),~ m \geq 0,
\]
\[
 \|z\|^p_{L^p(\mathcal{T}_h)} = \sum_{R \in \mathcal{T}_h}\|z\|_{L^p(R)},\,\,\forall\,z \in L^p(\mathcal{T}_h), 
\]
for all $1 \leq p < \infty$ and $\|z\|_{L^\infty(\mathcal{T}_h)} = \mbox{esssup}_{z\in\mathcal{T}_h}|z|$. For $F \in \Gamma_x, z_h \in X_h$
\[
\|z_h\|^2_{0,F} = \int_{F}[\![z_h]\!]_F\cdot[\![z_h]\!]_F\,{\rm d}s(x).
\]
For all $(\bm{w}_h,q_h) \in \bm{H}_h \times L_h$,
\[
\vertiii{\bm{w}_h}^2 = \|\nabla \bm{w}_h\|^2_{\bm{L}^2} + \sum_{F \in \Gamma_x}h^{-1}_x\|[\bm{w}_h]\|^2_{\bm{L}^2(F)},
\]
\[
\vertiii{\left(\bm{w}_h,q_h\right)}^2 = \vertiii{\bm{w}_h}^2 + \|q_h\|_{L^2(\mathcal{T}_h^x)}^2 + \sum_{F \in \Gamma_x}h_x\|[q_h]\|^2_{L^2(F)}.
\]
For our subsequent use, we recall some standard estimates. 
\begin{itemize}
	\item \textbf{Trace inequality:} (see \cite[Lemma 1.46, p. 27]{2}) Let $\phi_h \in X_h$, then for all $F \in \Gamma_x$ and $T^x \in \mathcal{T}_h^x$ we have
	\begin{equation}\label{traceeqn}
		\|\phi_h\|_{0,F } \leq Ch_x^{-\frac{1}{2}}\|\phi_h\|_{0,T^x}.
	\end{equation}
	\item \textbf{Norm comparison:} (see \cite[Lemma 1.50, p. 29]{2}) For $1\leq p,q \leq \infty$ and $\phi_h \in X_h$. There exists a constant $C > 0$ such that  
	\begin{equation}\label{normcompeqn}
		\|\phi_h\|_{L^p(T^x)} \leq Ch_x^{\frac{2}{p}-\frac{2}{q}}\|\phi_h\|_{L^q(T^x)} \quad \forall \quad T^x \in \mathcal{T}_h^x.
	\end{equation}
	\item \textbf{Inverse inequality:} (see \cite[Lemma 1.44, p. 26]{2}) If $\phi_h  \in X_h$. Then  
	\begin{equation}\label{eq:inverse}
		\|\nabla_x\phi_h\|_{0,T^x} \leq Ch_x^{-1}\|\phi_h\|_{0,T^x} \quad \forall \,\, T^x \in \mathcal{T}_h^x.
	\end{equation}
	\item A \textbf{Poincar\'e-Friedrichs} type inequality is proved in \cite[Lemma 2.1, p. 744]{arnold1982interior}  which says that
	\[
	\|w_h\|_{0,\Omega_x} \leq C\vertiii{w_h}, \quad \forall \quad w_h \in H^1(\Omega_x).
	\]
	\item \textbf{Projection operators :} Let $k \geq 0$. Let $\mathcal{P}_x : L^2(\Omega_x) \rightarrow X_h, \mathcal{P}_v : L^2(\Omega_v) \rightarrow V_h$, $\bm{\mathcal{P}_x} : \bm{L^2} \rightarrow \bm{H}_h, $ and $\mathcal{P}_h : L^2(\Omega) \rightarrow \mathcal{Z}_h$ be the standard $L^2$-projections onto the spaces $X_h, V_h, \bm{H}_h$ and $\mathcal{Z}_h$, respectively. Note that $\mathcal{P}_h = \mathcal{P}_x \otimes \mathcal{P}_v,$ (see \cite{ciarlet2002finite,agmon2010lectures}).\\
	By definition, $\mathcal{P}_h$ is stable in $L^2$-norm and it can be further shown to be stable in all $L^p$- norms (see \cite{crouzeix1987stability} for details). Let $1 \leq p \leq \infty$. Then for any $w \in L^p(\Omega)$
	\begin{equation}\label{L2stability}
		\|\mathcal{P}_h(w)\|_{L^p(\mathcal{T}_h)} \leq C\|w\|_{L^p(\Omega)}.
	\end{equation}
\end{itemize}

\subsection{Semi-discrete dG formulation}

Find $(f_h, \bm{u}_h,p_h)(t) \in \mathcal{Z}_h\times \bm{H}_h \times L_h $, for $t \in [0,T]$ such that\\
\begin{equation}{\label{bh}}
	\left(\frac{\partial f_h}{\partial t},\phi_h\right) + \mathcal{B}_{h}(\bm{u}_h ; f_{h},\phi_h) = 0  \hspace{2mm} \forall \,\,\phi_h \in \mathcal{Z}_h
\end{equation}
coupled with semi-discrete dG scheme for the Stokes system as follows:
\begin{equation}\label{ch1}
	\left(\frac{\partial \bm{u}_h}{\partial t}, \bm{\psi}_h\right) + a_h(\bm{u}_h,\bm{\psi}_h) + b_h(\bm{\psi}_h,p_h) + \left(\rho_h\bm{u}_h, \bm{\psi}_h\right) = \left(\rho_hV_h,\bm{\psi}_h\right) \quad \forall\,\, \bm{\psi}_h \in \bm{H}_h,
\end{equation}
\begin{equation}\label{ch2}
	-b_h(\bm{u}_h,w_h) + s_h(p_h,w_h) = 0 \quad \forall\,\,w_h \in L_h,
\end{equation}
with $f_h(0) = f_{0h} \in \mathcal{Z}_h$ and $\bm{u}_h(0) = \bm{u_{0h}} \in \bm{H}_h$  to be defined later,	
where
\begin{equation}
	\mathcal{B}_h(\bm{u}_h;f_h,\phi_h) := \sum_{R \in \mathcal{T}_h}\mathcal{B}_{h,R}(\bm{u}_h;f_h,\phi_h),
\end{equation}
with
\begin{equation}\label{bhdef}
	\begin{aligned}
		\mathcal{B}_{h,R}(\bm{u}_h;f_{h},\phi_h) & :=  - \int_{R}f_{h}v.\nabla_x\phi_h \, {\rm d}v\, {\rm d}x + \int_{T^v}\int_{\partial T^x}\widehat{v\cdot\bm{n}f_{h}}\phi_h\, {\rm d}s(x)\, {\rm d}v
		\\
		& \quad - \int_{R}f_h(\bm{u_{h}} - v).\nabla_v\phi_h \, {\rm d}v\, {\rm d}x +\int_{T^x}\int_{\partial T^v} \reallywidehat{(\bm{u}_h - v)\cdot\bm{n}f_h}\phi_h \, {\rm d}s(v)\, {\rm d}x,
	\end{aligned}
\end{equation}
wherein the numerical fluxes are taken to be
\begin{equation}\label{flux}
	\left\{
	\begin{aligned}
		\widehat{v\cdot \bm{n}f_h} &= \{vf_h\}_\beta\cdot\bm{n} := \left(\{vf_h\} + \frac{|v\cdot \bm{n}|}{2}[\![f_h]\!]\right)\cdot\bm{n}  \hspace{5mm} \text{on}\hspace{1mm} \Gamma^0_{x} \times T^v,
		\\
		\reallywidehat{(\bm{u}_h - v)\cdot \bm{n}f_h} &= \{(\bm{u}_h - v)f_h\}_\alpha\cdot\bm{n} := 
		\\
		&\qquad \left(\{(\bm{u}_h - v)f_h\} + \frac{|(\bm{u}_h - v)\cdot \bm{n}|}{2}[\![f_h]\!]\right)\cdot \bm{n}  \hspace{3mm} \text{on}\hspace{1mm} T^x \times \Gamma^0_{v},
	\end{aligned}
	\right.
\end{equation}
with $\bm{n} = \bm{n}^-, \beta = \frac{1}{2}(1\pm\mbox{sign}(v\cdot\bm{n}^\pm))$ and $\alpha = \frac{1}{2}(1\pm\mbox{sign}((\bm{u}_h-v)\cdot\bm{n}^\pm))$. For the more details about weighted average refer \cite{brezzi2004discontinuous}. On the boundary edges $e \in \Gamma_r^\partial, r = x,v$, we impose periodicity for $\reallywidehat{v\cdot\bm{n}f_h}$ and the compactness of support for $\reallywidehat{\left(\bm{u}_h - v\right)\cdot\bm{n}f_h}$. \\
In \eqref{ch1}, the bilinear form $a_h(\bm{u}_h,\bm{\psi}_h)$ stands for 
\begin{equation}{\label{aih}}
	\begin{aligned}
		a_h(\bm{u}_h,\bm{\psi}_h)  &= \sum_{i=1}^2a_{h,i}(u_{h,i},\psi_{h,i})
	\end{aligned}
\end{equation}
where, $u_{h,1}, u_{h,2}$ and $\psi_{h,1}, \psi_{h,2}$ denote the Cartesian components of $\bm{u}_h$ and $\bm{\psi}_h$, respectively with
\begin{equation*}
	\begin{aligned}
		a_{h,i}(u_{h,i},\psi_{h,i}) &= \sum_{T^x\in \mathcal{T}^x_{h}}\int_{T^x}\nabla{u_{h,i}}\cdot\nabla\psi_{h,i}\,{\rm d}x + \sum_{F\in \Gamma_x}\int_F\frac{\vartheta}{h_x}[\![u_{h,i}]\!]\cdot[\![\psi_{h,i}]\!]\,{\rm d}s(x)
		\\
		&\quad -\sum_{F\in \Gamma_x}\int_F\left(\{\nabla u_{h,i}\}\cdot [\![\psi_{h,i}]\!] + [\![u_{h,i}]\!]\cdot\{\nabla\psi_{h,i}\}\right)\,{\rm d}s(x).
	\end{aligned}
\end{equation*} 	
Here, $\vartheta > 0$ is a penalty parameter.

The $b_h(\cdot,\cdot)$ stands for 	
\begin{equation}{\label{bhh}}
	\begin{aligned}
		b_h(\bm{u}_h,w_h) = -\sum_{T^x\in \mathcal{T}^x_{h}}\int_{T^x} w_h\nabla\cdot  \bm{u}_h\,{\rm d}x + \sum_{F\in \Gamma_x}\int_F[\![\bm{u}_h]\!]\{w_h\}\,{\rm d}s(x)\, ,
	\end{aligned}
\end{equation}	
the stabilization term
\begin{equation}\label{sh}
	\begin{aligned}
		s_h(p_h,w_h) = \sum_{F\in\ \Gamma_x^0}h_x\int_F[\![p_h]\!]\cdot[\![w_h]\!]\,{\rm d}s(x).
	\end{aligned}
\end{equation}	
Like in continuous case, set discrete local density $\rho_h$ and discrete local macroscopic velocity $V_h$ as
\begin{equation}\label{rh}
	\rho_h = \sum_{T^v \in \mathcal{T}^v_{h}}\int_{T^v}f_h\,{\rm d}v \quad \mbox{and} \quad V_h = \frac{1}{\rho_h}\left(\sum_{T^v \in \mathcal{T}^v_{h}}\int_{T^v}vf_h\,{\rm d}v\right).
\end{equation}
Note that equations \eqref{ch1} and \eqref{ch2} are equivalent to 
\begin{equation}\label{ch}	
	\left(\frac{\partial \bm{u}_h}{\partial t}, \bm{\psi}_h\right) + \tilde{\mathcal{A}}((\bm{u}_h,p_h),(\bm{\psi}_h,w_h))\, + \,\left(\rho_h\bm{u}_h, \bm{\psi}_h\right) = \left(\rho_h V_h, \bm{\psi}_h\right) \hspace{2mm} \forall\hspace{1mm}(\bm{\psi}_h,w_h) \in \bm{H}_h\times L_h,
\end{equation}
where
\begin{equation}\label{A_h}
	\tilde{\mathcal{A}}((\bm{u}_h,p_h),(\bm{\psi}_h,w_h)) = a_h(\bm{u}_h,\bm{\psi}_h) + b_h(\bm{\psi}_h,p_h) - b_h(\bm{u}_h,w_h) + s_h(p_h,w_h).
\end{equation}
Note that equation \eqref{bhh} is equivalent to 
\begin{equation*}
	\begin{aligned}
		b_h(\bm{u}_h,w_h) = \sum_{T^x\in \mathcal{T}^x_{h}}\int_{T^x} \bm{u}_h\cdot\nabla  w_h\,{\rm d}x - \sum_{F\in \Gamma_x^0}\int_F[\![w_h]\!]\{\bm{u}_h\}\cdot \bm{n}_F\,{\rm d}s(x).
	\end{aligned}
\end{equation*}
The linear form $a_h(\bm{u}_h,\bm{\psi}_h)$ is coercive with respect to $\vertiii{\cdot}$-norm (for proof, see \cite[Lemma 4.12, p. 129]{2}), i.e. there exist $\beta > 0$ independent of $h$ such that
\begin{equation}\label{acoercive}
	\beta\vertiii{\bm{u}_h}^2 \leq a_h(\bm{u}_h,\bm{u}_h).
\end{equation}	
The bilinear form $\tilde{\mathcal{A}}((\bm{u}_h,p_h),(\bm{\psi}_h,w_h))$ satisfies the discrete inf-sup stability condition and is bounded in the $\vertiii{(\cdot,\cdot)}$-norm, i.e., there exist $\alpha > 0$ and $C > 0$ independent of $h$ such that
\begin{equation}\label{dic1}
	\alpha\vertiii{(\bm{u}_h,p_h)} \leq \sup_{(\bm{\psi}_h,w_h) \in \bm{H}_h\times L_h\setminus\{(0,0)\}}\frac{\tilde{\mathcal{A}}((\bm{u}_h,p_h),(\bm{\psi}_h,w_h))}{\vertiii{(\bm{\psi}_h,w_h)}}
\end{equation}
(for proof, refer \cite[ Lemma 6.13, p. 253]{2}) and 
\begin{equation}\label{ahbd}
	|\tilde{\mathcal{A}}((\bm{u}_h,p_h)(\bm{\psi}_h,w_h))| \leq C\vertiii{(\bm{u}_h,p_h)}\vertiii{(\bm{\psi}_h,w_h)}.
\end{equation}
Since $\bm{H}_h\times L_h \times \mathcal{Z}_h$ is finite dimensional, the discrete problem \eqref{bh} and \eqref{ch} leads to a system of non-linear ODEs. Then, an appeal to the Picard's theorem ensures an existence of a local-in-time unique solution $\left(\bm{u}_h,p_h,f_h\right)$. In order to continue the discrete solution for all $t \in [0,T]$, we need the boundedness of the discrete solution which we shall comment it as a Remark \ref{boundscase} later on in this paper. 

\subsection{Some properties of the discrete solution}\label{dgp}
Below we stated some properties satisfied by the discrete system. The proofs of which are similar to the proof of \cite[Lemma 3.3 - 3.6]{hutridurga2023discontinuous}. So we are skipping the details.
\begin{lem}
	Let $(f_h,\bm{u}_h,p_h) \in \mathcal{C}^1([0,T];\mathcal{Z}_h\times \bm{H}_h\times L_h)$ be the dG-dG approximation obtained as a solution to \eqref{bh} and \eqref{ch} with initial datum $f_h(0) = \mathcal{P}_h(0)$ and $\bm{u}_h(0) = \bm{\mathcal{P}_xu_0}$. Then, for all $t \geq 0$,
	\begin{itemize}
		\item Mass conservation
		\begin{equation*}
			\begin{aligned}
				\int_{\Omega}f_h(t,x,v)\,{\rm d}v\,{\rm d}x  = \int_{\Omega}f_0(x,v)\,{\rm d}v\,{\rm d}x \quad \forall\,\, k \geq 0.
			\end{aligned}
		\end{equation*}
		\item Momentum conservation
		\begin{equation*}
			\int_{\Omega}vf_h(t,x,v)\,{\rm d}v\,{\rm d}x + \int_{\Omega_x}\bm{u}_h\,{\rm d}x = \int_{\Omega}vf_0(x,v)\,{\rm d}v\,{\rm d}x + \int_{\Omega_x}\bm{u_0}\,{\rm d}x \quad \mbox{for}\quad k \geq 1.
		\end{equation*}
		\item For $k \geq 0$
		 \begin{equation}\label{eq:fhbd}
			\max_{t \in [0,T]}\|f_h\|_{0,\mathcal{T}_h} \leq e^{T}\|f_0\|_{0,\mathcal{T}_h}.
		\end{equation}
	\end{itemize}
\end{lem}

As a consequence of the above lemma, for a given non-negative initial data, we have
\[
\int_\Omega f_h(t,x,v){\rm d}x\,{\rm d}v \geq 0 \quad\mbox{and}\quad \int_{\Omega_x}\rho_h(t,x){\rm d}x \geq 0.
\]


Note that we do not have the discrete version of the energy dissipation given in Lemma \ref{lem:energy-dissipate} as it is difficult to prove non-negative property of $f_h$.


For our subsequent use, we shall need the following lemma.
\begin{lem}\label{rhoh2}
	Let $\rho$ and $\rho_h$ be the continuum and the  discrete local density associated to the $f$ and i $f_h$, respectively. Then,
	\begin{equation*}
		\|\rho - \rho_h\|_{0,\mathcal{T}^x_{h}} \leq 2L\|f - f_h\|_{0,\mathcal{T}_h} \quad \mbox{and} \quad \|\rho - \rho_h\|_{\infty,\mathcal{T}^x_{h}} \leq 4L^2\|f - f_h\|_{\infty,\mathcal{T}_h}.
	\end{equation*}
	Moreover, 		
	\begin{equation*}
		\|\rho V - \rho_h V_h\|_{0,\mathcal{T}^x_{h}} \leq 4L^2\|f - f_h\|_{0,\mathcal{T}_h}.
	\end{equation*}		
\end{lem}
For details of the proof  of the above lemma, refer \cite[Lemma 6]{hutridurga2024discontinuous}.
From equation \eqref{eq:fhbd}, it follows 
\begin{equation}\label{rho1}
	\|\rho_h\|_{0,\mathcal{T}^x_{h} }\leq C(T)L\|f_0\|_{0,\mathcal{T}_h}.
\end{equation}

\section{A priori Error Estimates}\label{err}
This section discusses some a priori error estimates for the discrete solution.
\subsection{Error estimates for Stokes system}

This subsection deals with error estimates for the Stokes system.

\textbf{Stokes projection:} Define the Stokes projection $\left(\bm{\Pi_uu}(t),\Pi_pp(t)\right)$ of $\left(\bm{u}(t),p(t)\right)$ for all $t \in [0,T]$ satisfying 
\begin{equation}\label{4.2}
	a_h(\bm{u} - \bm{\Pi_uu}, \bm{\psi}_h) + b_h(\bm{\psi}_h,p - \Pi_pp) = 0 \quad\forall\,\,\bm{\psi}_h \in \bm{H}_h,
\end{equation}
\begin{equation}\label{4.3}
	-b_h(\bm{u}-\bm{\Pi_uu},w_h) + s_h(p - \Pi_pp,w_h) = 0 \quad \forall\,\, w_h \in L_h.
\end{equation}
Systems \eqref{4.2} and \eqref{4.3} can be written as
\begin{equation}\label{stokesproj}
	\tilde{\mathcal{A}}\left(\left(\bm{\Pi_uu-u},\Pi_pp-p\right),\left(\bm{\psi}_h,w_h\right)\right) = 0, \quad \forall\,\, (\bm{\psi}_h,w_h) \in \bm{H}_h\times L_h
\end{equation}
where, $\tilde{\mathcal{A}}\left(\left(\cdot,\cdot\right),(\cdot,\cdot)\right)$ is defined in \eqref{A_h}.

From equation \eqref{dic1}, $\tilde{\mathcal{A}}\left((\bm{\Pi_uu},\Pi_pp),(\bm{\psi}_h,w_h)\right)$ satisfies the discrete inf-sup condition in the $\vertiii{\left(\cdot,\cdot\right)}$-norm and by equation \eqref{ahbd}, it is bounded from above in the $\vertiii{\left(\cdot,\cdot\right)}$-norm. Therefore, an application of the Lax-Milgram lemma shows existence of a unique pair $(\bm{\Pi_uu},\Pi_pp) \in \bm{H}_h \times L_h$.

The following lemma gives the approximation properties of the Stokes projection, for proof refer \cite[Corollary 6.26, p. 260]{2}.
\begin{lem}
	Let $(\bm{\Pi_uu},\Pi_pp) \in \bm{H}_h\times L_h$ solve \eqref{stokesproj}. Assume that $(\bm{u},p) \in L^\infty(0,T;\bm{H}^{k+1}) \times L^\infty(0,T;H^k)$. Then,
	\begin{equation}\label{uestimat}
		\|\bm{u - \Pi_uu}\|_{\bm{L}^2} + h_x\vertiii{\bm{u - \Pi_uu}} + h_x\|p - \Pi_pp\|_{0,\mathcal{T}_h^x} \leq C h_x^{k+1},
	\end{equation}
	where, $C$ is a positive constant which is independent of $h_x$.
\end{lem}

Next, lemma shows a  relation between the $L^\infty$ and $L^2$ bounds while approximating function in the broken polynomial space. For the proof refer \cite[Lemma 3, p. 5]{hutridurga2024discontinuous}.
\begin{lem}\label{L:uinf}
	Let $\bm{u}_h \in \bm{H}_h$, an approximation of $\bm{u}$ be defined by \eqref{ch1}-\eqref{ch2}. Assume that $\bm{u} \in \bm{W}^{1,\infty} \cap \bm{H}^{k+1}$. Then, 
	\begin{equation*}\label{inftybound}
		\|\bm{u} - \bm{u}_h\|_{\bm{L}^\infty} \lesssim h_x\|\bm{u}\|_{\bm{W}^{1,\infty}} + h^k_x\|\bm{u}\|_{\bm{k+1,2}} + h^{-1}_x\|\bm{u} - \bm{u}_h\|_{\bm{L}^2}.
	\end{equation*}
\end{lem}


\textbf{Error equation for Stokes part:}
Since the scheme \eqref{ch1}-\eqref{ch2} is consistent, \eqref{ch1}-\eqref{ch2} also hold for the solution $(\bm{u},p,f)$ to the continuum model. Hence, by taking the difference with $\bm{e_u} := \bm{u} - \bm{u}_h,~ e_p := p - p_h$, we obtain the following equations:
\begin{equation}\label{errstokes11}
	\begin{aligned}
		\left(\frac{\partial \bm{e_u}}{\partial t},\bm{\psi}_h\right) + a_h\left(\bm{e_u},\bm{\psi}_h\right) + b_h\left(\bm{\psi}_h,e_p\right) &+ \left(\rho\bm{u} - \rho_h\bm{u}_h , \bm{\psi}_h\right) 
		\\
		&= \left(\rho V - \rho_hV_h,\bm{\psi}_h\right), \,\, \forall\,\,\bm{\psi}_h \in \bm{H}_h,
	\end{aligned}
\end{equation}
\begin{equation}\label{errstokes21}
	-b_h(\bm{e_u},w_h) + s_h(e_p,w_h) = 0 \quad \forall\,\, w_h \in L_h.
\end{equation}
Using the Stokes projection operator, we rewrite
\begin{equation}\label{ues}
	\begin{aligned}
		\bm{e_u} &:= \bm{u} - \bm{u}_h := \left(\bm{\Pi_uu} - \bm{u}_h\right) -  \left(\bm{\Pi_uu} - \bm{u}\right) =: \bm{\theta_u - \eta_u},
		\\
		e_p &:= p - p_h := (\Pi_pp - p_h) - (\Pi_pp - p) =: \theta_p - \eta_p.
	\end{aligned}
\end{equation}
Using \eqref{ues} and \eqref{4.2}-\eqref{4.3}, the error equation \eqref{errstokes11}-\eqref{errstokes21} becomes
\begin{equation}\label{errstokes1}
	\begin{aligned}
		\left(\frac{\partial \bm{\theta_u}}{\partial t},\bm{\psi}_h\right) & + a_h\left(\bm{\theta_u},\bm{\psi}_h\right) + b_h\left(\bm{\psi}_h,\theta_p\right) + \left(\rho\bm{\theta_u}, \bm{\psi}_h\right) = \left(\rho V - \rho_hV_h,\bm{\psi}_h\right) 
		\\
		&+ \left(\frac{\partial\bm{\eta_u}}{\partial t}, \bm{\psi}_h\right) + \left(\rho\bm{\eta_u},\bm{\psi}_h\right) + \left(\left(\rho - \rho_h\right)\bm{\theta_u},\bm{\psi}_h\right) 
		\\
		&- \left(\left(\rho - \rho_h\right)\bm{\eta_u},\bm{\psi}_h\right) - \left(\left(\rho - \rho_h\right)\bm{u},\bm{\psi}_h\right), \quad \forall\,\,\bm{\psi}_h \in \bm{H}_h
	\end{aligned}
\end{equation}
\begin{equation}\label{errstokes2}
	-b_h(\bm{\theta_u},w_h) + s_h(\theta_p,w_h) = 0 \quad \forall\,\, w_h \in L_h.
\end{equation}


\begin{lem}\label{ubound}
	Let $(\bm{u},p)$ be the unique solution of the Stokes equation \eqref{contstokes}. Let $(\bm{u}_h,p_h) \in \bm{H}_h\times L_h$ solve \eqref{ch}. Assume $(\bm{u},p) \in L^\infty(0,T;\bm{H}^{k+1}) \times L^\infty(0,T;H^k)$. Then, there exists a positive constant $C$ independent of $h$, such that for all $t \in (0,T]$,
	\begin{equation}\label{thetauest}
		\begin{aligned}
			\frac12\frac{{\rm d}}{{\rm d}t}\|\bm{\theta_u}\|^2_{\bm{L}^2} &+ \beta\vertiii{\bm{\theta_u}(t)}^2 + s_h(\theta_p,\theta_p) 
			\\
			&\leq C\left(h^{k+1} + \|f - f_h\|_{0,\mathcal{T}_h} + h^{-1}\|f - f_h\|_{0,\mathcal{T}_h}\|\bm{\theta_u}\|_{\bm{L}^2}\right)\|\bm{\theta_u}\|_{\bm{L}^2}.
		\end{aligned}
	\end{equation}
\end{lem}

\begin{proof}
	Choose $\bm{\psi}_h = \bm{\theta_u}$ and $w_h = \theta_p$ in \eqref{errstokes1} and \eqref{errstokes2}, respectively. Then add the resulting expressions. Using the H\"older inequality and the coercivity \eqref{acoercive} of $a_h(\bm{\theta_u,\theta_u})$, we obtain
	\begin{equation*}
		\begin{aligned}
			\frac{1}{2}\frac{{\rm d}}{{\rm d}t}\|\bm{\theta_u}\|^2_{\bm{L}^2} + \beta\vertiii{\bm{\theta_u}(t)}^2 &+ s_h(\theta_p,\theta_p)  + \|\rho^{\frac{1}{2}}\bm{\theta_u}\|^2_{\bm{L}^2} \leq \left( \|\rho V - \rho_hV_h\|_{0,\mathcal{T}_h^x} + \|\partial_t\bm{\eta_u}\|_{\bm{L}^2} \right.
			\\
			&\left.+ \|\rho\|_{L^\infty(\Omega_x)}\|\bm{\eta_u}\|_{\bm{L}^2} + \|\rho - \rho_h\|_{0,\mathcal{T}^x_h}\|\bm{u}\|_{\bm{L}^\infty}\right)\|\bm{\theta_u}\|_{\bm{L}^2}
			\\
			&+ \left(\|\rho - \rho_h\|_{0,\mathcal{T}_h^x}\left(\|\bm{\theta_u}\|_{\bm{L}^2} + \|\bm{\eta_u}\|_{\bm{L}^2}\right)\right)\|\bm{\theta_u}\|_{\bm{L}^\infty}.
		\end{aligned}
	\end{equation*}
	Since the last term on the left hand side is non-negative, it can be dropped. A use of the projection estimate along with the Lemma \ref{rhoh2} and the estimate $\|\cdot\|_{L^\infty} \leq h^{-1}\|\cdot\|_{L^2}$ yields \eqref{thetauest}.
\end{proof}

\subsection{Error estimates for the kinetic equation }

Since the scheme \eqref{bh} is consistent, the solution $(\bm{u},f)$ to the continuum problem should satisfy 
\begin{equation}\label{weakvlasov}
	\left(\frac{\partial f}{\partial t},\phi_h\right) + \mathcal{B}(\bm{u};f,\phi_h) = 0 \quad\forall\,\phi_h \in \mathcal{Z}_h,
\end{equation}
where 
\[
\mathcal{B}(\bm{u};f,\phi_h) := \sum_{R \in \mathcal{T}_{h}}\mathcal{B}_{R}(\bm{u};f,\phi_h),
\]
with
\begin{equation*}
	\begin{aligned}
		\mathcal{B}_{R}(\bm{u};f,\phi_h) &=  - \int_R f v\cdot \nabla_x\phi_h\,{\rm d}x\,{\rm d}v - \int_R f(\bm{u}-v)\cdot\nabla_v\phi_h\,{\rm d}x\,{\rm d}v 
		\\
		&+ \int_{T^v}\int_{\partial T^x}v\cdot \bm{n} f\phi_h\,{\rm d}s(x)\,{\rm d}v + \int_{\partial T^v}\int_{T^x}(\bm{u}-v)\cdot \bm{n} f\phi_h\,{\rm d}x\,{\rm d}s(v).  
	\end{aligned}
\end{equation*}	
Subtracting equation \eqref{bh} from \eqref{weakvlasov}, we obtain the error equation:
\begin{equation*}\label{erro}
	\begin{aligned}
		\left(\frac{\partial }{\partial t}\left(f - f_h\right),\phi_h\right) + \mathcal{B}(\bm{u};f,\phi_h) - \mathcal{B}_{h}(\bm{u}_h;f_h,\phi_h) = 0
		\quad\forall\, \phi_h \in \mathcal{Z}_h.  
	\end{aligned}
\end{equation*}
Setting $e_f = f - f_h$, we rewrite the error equation as  
\begin{equation}\label{error}
	\begin{aligned}
		\left(\frac{\partial e_f}{\partial t},\phi_h\right) + a_h^0(e_f,\phi_h) + \mathcal{N}(\bm{u};f,\phi_h) - \mathcal{N}^h(\bm{u}_h;f_h,\phi_h) = 0
		\quad\forall\, \phi_h \in \mathcal{Z}_h , 
	\end{aligned}
\end{equation}
where
\begin{equation*}\label{ae}
	a_h^0(e_f,\phi_h) = -\sum_{R \in \mathcal{T}_{h}}\int_{R}e_f v\cdot \nabla_x\phi_h\,{\rm d}x\,{\rm d}v
	-\sum_{T^v \in \mathcal{T}_{h}^v}\int_{T^v}\int_{\Gamma_x}\{v e_f\}_\beta\cdot[\![\phi_h]\!]\,{\rm d}s(x)\,{\rm d}v,
\end{equation*}	
\begin{equation}\label{N}
	\mathcal{N}(\bm{u};f,\phi_h) = -\sum_{R \in \mathcal{T}_{h}}\int_{R} f (\bm{u} - v)\cdot\nabla_v\phi_h\,{\rm d}x\,{\rm d}v \,\, - \sum_{T^x \in \mathcal{T}_{h}^x}\int_{\Gamma_v}\int_{T^x}(\bm{u} - v)f\cdot[\![\phi_h]\!]\,{\rm d}x\,{\rm d}s(v) \, ,
\end{equation}
and
\begin{equation}\label{Nh}
	\begin{aligned}
		\mathcal{N}^h(\bm{u}_h;f_h,\phi_h) = &-\sum_{R \in \mathcal{T}_{h}}\int_{R} f_h (\bm{u}_h - v)\cdot\nabla_v\phi_h\,{\rm d}x\,{\rm d}v 
		\\
		&- \sum_{T^x \in \mathcal{T}_{h}^x}\int_{\Gamma_v}\int_{T^x}\{(\bm{u}_h - v) f_h\}_\alpha\cdot[\![\phi_h]\!]\,{\rm d}x\,{\rm d}s(v) \, .
	\end{aligned}
\end{equation}

\subsection{Special Projection:}
Recall now the projection operator defined in \cite{hutridurga2024discontinuous}. We set projection operator $\Pi_h : \mathcal{C}^0(\Omega) \rightarrow \mathcal{Z}_h$ as follows: For an arbitrary element $R = T^x \times T^v$  of $\mathcal{T}_h$ and for any $w \in \mathcal{C}^0({R})$. The restriction of $\Pi_h w$ to $R$ is given by 
\begin{equation}\label{PIi}
	\Pi_h(w) =
	\left\{
	\begin{aligned}
		&(\tilde{\Pi}_x\otimes\tilde{\Pi}_v)w\quad\mbox{if sign}\left(\left(\bm{u} - v\right)\cdot \bm{n}\right) = \mbox{constant},
		\\
		&(\tilde{\Pi}_x \otimes \tilde{\mathcal{P}}_v)w \quad\mbox{if sign}\left(\left(\bm{u} - v\right)\cdot \bm{n}\right) \neq \mbox{constant},
	\end{aligned}
	\right.
\end{equation}
where $\tilde{\Pi}_x : \mathcal{C}^0(\Omega_x) \rightarrow X_h$ and $\tilde{\Pi}_v : \mathcal{C}^0(\Omega_v) \rightarrow V_h$ are the $2$-dimensional projection operators, see \cite[Section 4.3, p. 12]{hutridurga2024discontinuous} for details.

The projection operator $\Pi_h$ satisfies the following approximation properties, whose proof is similar to \cite[Lemma 4.1, p. 9]{de2012discontinuous}.
\begin{lem}\label{westimate}
	Let $w \in H^{k+1}(R), k \geq 0$ and let $\Pi_h$ be the projection operator defined through \eqref{PIi}. Then, for all $e \in (\partial T^x \times T^v) \cup (T^x \times \partial T^v)$, 
	\begin{equation}\label{w}
		\|w - \Pi_h w\|_{0,R} + h^\frac{1}{2}\|w - \Pi_h w\|_{0,e} \leq Ch^{k+1}\|w\|_{k+1,R}.
	\end{equation}
\end{lem}

Summing estimates \eqref{w}, over all elements of the partition $\mathcal{T}_h$, we obtain
\begin{equation}\label{wprojection}
	\begin{aligned}
		\|w - \Pi_h w\|_{0,\mathcal{T}_h} &+ h^{\frac{1}{2}}\|w - \Pi_h w\|_{0,\Gamma_x\times \mathcal{T}^v_{h}}
		\\
		& + h^{\frac{1}{2}}\|w - \Pi_h w\|_{0,\mathcal{T}^x_{h}\times\Gamma_v} \leq Ch^{k+1}\|w\|_{k+1,\Omega}.
	\end{aligned}
\end{equation}


Using the special projection, split $f - f_h$ as
\begin{equation}\label{p}
	e_f := f - f_h := \left(\Pi_h f - f_h\right) - \left(\Pi_h f - f\right) := \theta_f - \eta_f.
\end{equation}

\begin{lem}\label{lemn}
	Let $f \in C^0(\Omega), \bm{u} \in C^0(\Omega_x)$ and $f_h \in \mathcal{Z}_h$. Then, the following identity holds true
	\begin{align*}
		\mathcal{N}(\bm{u};f,\theta_f) - &\mathcal{N}^h(\bm{u}_h;f_h,\theta_f) = \sum_{T^x \in \mathcal{T}_{h}^x}\int_{\Gamma_v}\int_{T^x}\frac{|(\bm{u}_h - v)\cdot \bm{n}|}{2}\vert[\![\theta_f]\!]\vert^2\,{\rm d}x\,{\rm d}s(v)
		\\
		& + \sum_{R \in \mathcal{T}_{h}}\int_{R} \left((\bm{u} - \bm{u}_h)\cdot\nabla_v f\,\theta_f - \theta_f^2\right)\,{\rm d}x\,{\rm d}v + \mathcal{K}^2(\bm{u}_h - v,f,\theta_f),
	\end{align*}
	where
	\begin{equation}\label{K2}
		\begin{aligned}
			\mathcal{K}^2(\bm{u}_h - v,f,\theta_f) &= \sum_{R \in \mathcal{T}_{h}}\int_R \eta_f (\bm{u}_h - v)\cdot \nabla_v \theta_f\,{\rm d}x\,{\rm d}v 
			\\
			&\qquad+ \sum_{T^x \in \mathcal{T}_{h}^x}\int_{\Gamma_v}\int_{T^x}\{\left((\bm{u}_h - v) \eta_f\right)\}_\alpha\cdot[\![\theta_f]\!]\,{\rm d}x\,{\rm d}s(v)\, .
		\end{aligned}
	\end{equation}
\end{lem}

\begin{proof}
The proof is similar to the \cite[Lemma 11, p.13]{hutridurga2024discontinuous}.
\end{proof}
For the estimation on $e_f := \theta_f - \eta_f$, it is enough to estimate $\theta_f$ as the estimate of $\eta_f$ is already known from Lemma \ref{westimate}. After choosing $\phi_h = \theta_f$ in error equation \eqref{error} with a use of equation \eqref{p} and Lemma \ref{lemn}, we rewrite equation \eqref{error} in $\theta_f$ as

\begin{equation}\label{err1}
	\begin{aligned}
		\frac{1}{2}& \frac{{\rm d}}{{\rm d}t}\|\theta_f\|^2_{0,\mathcal{T}_h} + \sum_{T^v \in \mathcal{T}_{h}^v}\int_{T^v}\int_{\Gamma_x}\frac{|v\cdot \bm{n}|}{2}\vert[\![\theta_f]\!]\vert^2\,{\rm d}s(x)\,{\rm d}v  
		\\
		&+ \sum_{T^x \in \mathcal{T}_{h}^x}\int_{T^x}\int_{\Gamma_v}\frac{|(\bm{u}_h - v)\cdot \bm{n}|}{2}\vert[\![\theta_f]\!]\vert^2\,{\rm d}s(v)\,{\rm d}x = \left(\partial_t\eta_f, \theta_f\right) - \mathcal{K}^1(v,\eta_f,\theta_f) 
		\\
		&\qquad- \sum_{R \in \mathcal{T}_{h}}\int_{R} \left((\bm{u} - \bm{u}_h)\cdot\nabla_v f\,\theta_f - \theta^2_f\right)\,{\rm d}v\,{\rm d}x - \mathcal{K}^2(\bm{u}_h - v,f,\theta_f) ,
	\end{aligned}
\end{equation}	
where,
\begin{equation}\label{K1}
	\mathcal{K}^1(v,\eta_f,\theta_f) = \sum_{R \in \mathcal{T}_h}\int_{R} \eta_f v\cdot\nabla_x \theta_f\,{\rm d}x\,{\rm d}v + \sum_{T^v \in \mathcal{T}_{h}^v}\int_{T^v}\int_{\Gamma_x}\{v \eta_f\}_\beta\cdot[\![\theta_f]\!]\,{\rm d}s(x)\,{\rm d}v\, .
\end{equation}


Next, we will state the results for the estimates of $\mathcal{K}^1$ and $\mathcal{K}^2$, whose proof is similar to the proof of \cite[Lemma 12-13, p. 14]{hutridurga2024discontinuous}.

\begin{lem}\label{K1estimate}
	Let $k \geq 1$ and let 
	\[
	f \in C^0([0,T]; W^{1,\infty}(\Omega) \cap H^{k+2}(\Omega))
	\]
	 be the solution of \eqref{eq:continuous-model}. Further, let $f_h(t) \in \mathcal{Z}_h$ be its approximation satisfying \eqref{bh} and let $\mathcal{K}^1(v,\eta_f,\theta_f)$ be defined as in \eqref{K1}. Assume that the partition $\mathcal{T}_h$ is constructed so that none of the components of $v$ vanish inside any element. Then, the following estimates holds:
	\begin{equation*}\label{k1estimate}
		|\mathcal{K}^1(v,\eta_f,\theta_f)| \leq Ch^{k+1}\left(\|f\|_{k+1,\Omega} + L\|f\|_{k+2,\Omega}\right)\|\theta_f\|_{0,\mathcal{T}_h},
	\end{equation*} 
	for all $t \in [0,T]$.
\end{lem}



\begin{lem}\label{K2estimate}
	Let $k \geq 1$ and let $(f_h, \bm{u}_h) \in \mathcal{Z}_h\times \bm{H}_h$ be the solution to \eqref{bh}. Let 
	\[
	(\bm{u},f) \in L^\infty([0,T]; \bm{W}^{1,\infty}\cap \bm{H}^{k+1})\times L^\infty([0,T]; W^{1,\infty}(\Omega)\cap H^{k+2}(\Omega))
	\]
	 and let $\mathcal{K}^2$ is defined as in \eqref{K2}. Then, the following estimate holds
	\begin{equation}\label{k2estimate}
		\begin{aligned}
			|\mathcal{K}^2(\bm{u}_h - v,f,\theta_f)&| \leq Ch^k\|\bm{u}_h - \bm{u}\|_{\bm{L}^\infty}\|f\|_{k+1,\Omega}\|\theta_f\|_{0,\mathcal{T}_h} 
			\\
			&+ Ch^{k+1}\left(\|\bm{u}\|_{\bm{W}^{1,\infty}}\|f\|_{k+1,\Omega} + \left(L + \|\bm{u}\|_{\bm{L}^\infty}\right)\|f\|_{k+2,\Omega}\right)\|\theta_f\|_{0,\mathcal{T}_h}.
		\end{aligned}
	\end{equation}		
\end{lem}


\subsection{Optimal error estimates:}

\begin{lem}\label{fl2}
	For $k \geq 1$, let 
	\[
	f \in \mathcal{C}^1([0,T];H^{k+2}(\Omega)\cap W^{1,\infty}(\Omega))
	\]
	 be the solution of the Vlasov-Stokes problem \eqref{eq:continuous-model}-\eqref{contstokes} and let 
	 \[
	 \bm{u} \in \mathcal{C}^0([0,T];\bm{H}^{k+1}\cap \bm{W}^{1,\infty})
	 \]
	  be the associated fluid velocity. Further, let $(\bm{u}_h,f_h) \in \bm{H}_h \times \mathcal{Z}_h $ be the dG-dG approximation solution of \eqref{bh} and \eqref{ch}. Then there holds
	\begin{equation}\label{thetafest}
		\begin{aligned}
			\frac{{\rm d}}{{\rm d}t}\|\theta_f\|_{0,\mathcal{T}_h} \leq C_{\bm{u},f}h^{k+1} + C_f\|\bm{u} - \bm{u}_h\|_{\bm{L}^2} + \|\theta_f\|_{0,\mathcal{T}_h},
		\end{aligned}
	\end{equation}
where $C$ depends on the final time $T$, the polynomial degree $k$, the shape regularity of the partition and also regularity result of $(\bm{u},f)$. 
\end{lem}

\begin{proof}
From equation \eqref{err1}, there follows
\begin{equation}\label{estimate}
	\begin{aligned}
		\frac{1}{2}\frac{{\rm d}}{{\rm d}t}\|\theta_f\|^2_{0,\mathcal{T}_h} &+ \frac{1}{2}\||v\cdot \bm{n}|^\frac{1}{2}[\![\theta_f]\!]\|^2_{\Gamma_x\times \mathcal{T}^v_{h}} + \frac{1}{2}\||(\bm{u}_h - v)\cdot \bm{n}|^\frac{1}{2}[\![\theta_f]\!]\|^2_{\mathcal{T}^x_{h}\times \Gamma_v}
		\\
		& = I_1 - I_2 - \mathcal{K}^1 - \mathcal{K}^2 + I_3 \leq |I_1| + |I_2| + |\mathcal{K}^1| + |\mathcal{K}^2| + |I_3|.
	\end{aligned}
\end{equation}
Here, we use the standard triangle inequality and
\[
I_1 = \left(\partial_t\eta_f, \theta_f\right), \,\, I_2 = \sum_{R \in \mathcal{T}_{h}}\int_R(\bm{u} - \bm{u}_h)\cdot \nabla_v f \,\theta_f\,{\rm d}v\,{\rm d}x
\,\,\mbox{and}\,\,
I_3 = \sum_{R \in \mathcal{T}_{h}}\int_R\theta^2_f\,{\rm d}v\,{\rm d}x.
\]
First estimate $I_1$, a use of the Cauchy-Schwarz inequality with \eqref{wprojection} yields
\begin{equation}\label{I1estimate}
	\begin{aligned}
		|I_1| \leq \|\partial_t\eta_f\|_{0,\mathcal{T}_h}\|\theta_f\|_{0,\mathcal{T}_h} &\leq Ch^{k+1}\|f_t\|_{k+1,\mathcal{T}_h}\|\theta_f\|_{0,\mathcal{T}_h}
	\end{aligned}
\end{equation}
For $I_2$, 
\begin{equation}\label{I2estimate}
	\begin{aligned}
		|I_2| = \Big\vert\sum_{R \in \mathcal{T}_{h}}\int_R(\bm{u} - \bm{u}_h)\cdot \nabla_v f \theta_f\,{\rm d}v\,{\rm d}x\Big\vert \leq \|\bm{u} - \bm{u}_h\|_{\bm{L}^2}\|\nabla_v f\|_{L^\infty(\Omega)}\|\theta_f\|_{0,\mathcal{T}_h}
	\end{aligned}
\end{equation}
here, in the second step use the H\"older inequality.\\
Now, $\mathcal{K}^1$ is estimated in Lemma \ref{K1estimate},
\begin{equation}\label{K1Estimate}
	|\mathcal{K}^1| \leq Ch^{k+1}\left(\|f\|_{k+1,\Omega} + L\|f\|_{k+2,\Omega}\right)\|\theta_f\|_{0,\mathcal{T}_h}.
\end{equation}
To deal with $\mathcal{K}^2$, we observe that the bound \eqref{k2estimate} in Lemma \ref{K2estimate} yields
\begin{equation}\label{K2Estimate}
	\begin{aligned}
		|\mathcal{K}^2| &\leq Ch^k\|\bm{u} - \bm{u}_h\|_{\bm{L}^\infty}\|f\|_{k+1,\Omega}\|\theta_f\|_{0,\mathcal{T}_h} 
		\\
		&\quad + Ch^{k+1}\left(\|\bm{u}\|_{\bm{W}^{1,\infty}}\|f\|_{k+1,\Omega} + \left(L + \|\bm{u}\|_{\bm{L}^\infty}\right)\|f\|_{k+2,\Omega}\right)\|\theta_f\|_{0,\mathcal{T}_h}
		\\
		&\leq Ch^k\left(h\|\bm{u}\|_{\bm{W}^{1,\infty}} + h^k\|\bm{u}\|_{\bm{H}^{k+1}} + h^{-1}\|\bm{u} - \bm{u}_h\|_{\bm{L}^2}\right)\|f\|_{k+1,\Omega}\|\theta_f\|_{0,\mathcal{T}_h} 
		\\
		&\quad + Ch^{k+1}\left(\|\bm{u}\|_{\bm{W}^{1,\infty}}\|f\|_{k+1,\Omega} + \left(L + \|\bm{u}\|_{\bm{L}^\infty}\right)\|f\|_{k+2,\Omega}\right)\|\theta_f\|_{0,\mathcal{T}_h}
		\\
		&\leq \left(C_{\bm{u},f} h^{k+1} + C_f\|\bm{u} - \bm{u}_h\|_{\bm{L}^2}\right)\|\theta_f\|_{0,\mathcal{T}_h}.
	\end{aligned}
\end{equation}
Here, in second step we use Lemma \ref{L:uinf}.\\
Now, substituting the estimates \eqref{I1estimate}-\eqref{K2Estimate} into \eqref{estimate} and using the fact that the last two terms on the left hand side of the equation \eqref{estimate} is non-negative, we obtain \eqref{thetafest} and this completes the proof.
\end{proof}


\begin{thm}\label{fi}
Let $k \geq 1$ and let 
\[
(\bm{u},p) \in \mathcal{C}^1([0,T];\bm{H}^{k+1}\cap \bm{W}^{1,\infty}) \times \mathcal{C}^0([0,T];H^k(\Omega_x)) 
\]
\[
f \in \mathcal{C}^1([0,T];H^{k+2}(\Omega)\cap W^{1,\infty}(\Omega)) 
\]
 be the solution of the Vlasov-Stokes equation \eqref{eq:continuous-model}-\eqref{contstokes}. Let 
 \[
 (\bm{u}_h,p_h,f_h) \in \mathcal{C}^0([0,T];\bm{H}_h\times L_h) \times \mathcal{C}^1([0,T];\mathcal{Z}_h)
 \]
  be the dG-dG approximation solutions of \eqref{bh} and\eqref{ch}, then 
\[
\|f(t) - f_h(t)\|_{L^\infty(0,T;L^2(\Omega))} + \|(\bm{u} - \bm{u}_h)(t)\|_{L^\infty(0,T;\bm{L}^2)} \leq Ch^{k+1} \quad \forall\,\,t \in [0,T]
\] 
where $C$ depends on the final time $T$, the polynomial degree $k$, the shape regularity of the partition and norm of $(\bm{u},f)$.
\end{thm}

\begin{proof}
A use of equations \eqref{ues} and \eqref{p} with triangle inequality implies
\begin{equation*}
	\begin{aligned}
		\|\bm{u} - \bm{u}_h\|_{\bm{L}^2} + \|f - f_h\|_{0,\mathcal{T}_h} \leq \|\bm{\eta_u}\|_{\bm{L}^2} + \|\bm{\theta_u}\|_{\bm{L}^2} + \|\eta_f\|_{0,\mathcal{T}_h} + \|\theta_f\|_{0,\mathcal{T}_h}.
	\end{aligned}
\end{equation*}
Now, from equations  \eqref{uestimat} and \eqref{wprojection}, the estimates of $\|\bm{\eta_u}\|_{\bm{L}^2}$ and $\|\eta_f\|_{0,\mathcal{T}_h}$ are known, therefore, it is enough to estimate $\|\bm{\theta_u}\|_{\bm{L}^2} + \|\theta_f\|_{0,\mathcal{T}_h}$.\\
Adding \eqref{thetauest} and \eqref{thetafest} shows
\begin{equation}\label{a1}
	\begin{aligned}
		\frac{{\rm d}}{{\rm d}t}\left(\|\bm{\theta_u}\|_{\bm{L}^2} + \|\theta_f\|_{0,\mathcal{T}_h}\right) &\lesssim  h^{k+1} + \|\bm{\theta_u}\|_{\bm{L}^2} + \|\theta_f\|_{0,\mathcal{T}_h} + h^{-1}\|\theta_f\|_{0,\mathcal{T}_h}\|\bm{\theta_u}\|_{\bm{L}^2}
		\\
		&\lesssim  h^{k+1} + \|\bm{\theta_u}\|_{\bm{L}^2} + \|\theta_f\|_{0,\mathcal{T}_h} + h^{-1}\|\theta_f\|^2_{0,\mathcal{T}_h} + h^{-1}\|\bm{\theta_u}\|^2_{\bm{L}^2}.
	\end{aligned}
\end{equation}
Setting 
\begin{equation}\label{a11}
	\vertiii{(\bm{\theta_u},\theta_f)} := \|\bm{\theta_u}\|_{\bm{L}^2} + \|\theta_f\|_{0,\mathcal{T}_h}
\end{equation}
and a function $\Psi$ as
\begin{equation}\label{a111}
	\Psi(t) = h^{k+1} + \int_0^t\left(\vertiii{(\bm{\theta_u},\theta_f)(s)} + h^{-1}\vertiii{(\bm{\theta_u},\theta_f)(s)}^2\right){\rm d}s.
\end{equation}
An integration of equation \eqref{a1} with respect to time from $0$ to $t$ with \eqref{a11}-\eqref{a111} yields
\begin{equation*}\label{a1111}
	\vertiii{(\bm{\theta_u},\theta_f)(t)} \leq C\Psi(t).
\end{equation*}
Without loss of generality assume that $\vertiii{(\bm{\theta_u},\theta_f)(t)} > 0$, otherwise, we add an arbitrary small quantity say $\delta$ and proceed as in a similar way as describe below and then pass the limit as $\delta \rightarrow 0$. Note that $0 < \Psi(0) \leq \Psi(t)$ and $\Psi(t)$ is differentiable. An differentiation of $\Psi(t)$ with respect to $t$, shows
\begin{equation*}
	\begin{aligned}
		\partial_t\Psi(t) &= \vertiii{(\bm{\theta_u},\theta_f)} + h^{-1}\vertiii{(\bm{\theta_u},\theta_f)}^2
		\\
		&\leq C\left(\Psi(t) + h^{-1}(\Psi(t))^2\right).
	\end{aligned}
\end{equation*}
Moreover, $\partial_t\Psi(t) > 0$ and hence, $\Psi(t)$ is strictly monotonically increasing function which is also positive. An integration in time yields 
\begin{equation*}
	\int_0^t\frac{\partial_s\Psi(s)}{\Psi(s)\left(1 + h^{-1}\Psi(s)\right)}\,{\rm d}s \leq \int_0^tC\,{\rm d}s \leq CT.
\end{equation*} 
After evaluating integration on the left hand side exactly by using $\Psi(0) = h^{k+1}$ and taking exponential both side, we obtain
\begin{equation*}
	\Psi(t)\left(1 - h^{k}\left(e^{CT} - 1\right)\right) \leq e^{CT}h^{k+1}.
\end{equation*}
Now, choose small $h > 0$ such that $\left(1 - h^{k}\left(e^{CT} - 1\right)\right) > 0$ and this implies
\begin{equation*}
	\Psi(t) \leq Ch^{k+1}.
\end{equation*}
This completes the proof.
\end{proof}

The idea to define function $\Psi$ and applying non-linear Gr\"onwall's inequality in the proof of above Theorem is inspired by Lemma $4.2$ in \cite{ploymaklam2017priori}.

\begin{lem}\label{utbound}
	Let the hypothesis of Theorem \ref{fi} holds.
Then, there exists a positive constant $C$ independent of $h$, such that for all $t \in (0,T]$
\begin{equation*}\label{thetautest}
	\begin{aligned}
		\|\partial_t\bm{\theta_u}\|^2_{L^2(0,T;\bm{L}^2)} &+ \vertiii{\bm{\theta_u}(t)}^2 + s_h(\theta_p,\theta_p) \leq C h^{2(k+1)}.
	\end{aligned}
\end{equation*}
\end{lem} 

\begin{proof}
A differentiation of equation \eqref{errstokes2} with respect to time with a choice $w_h = \theta_p$ shows
\begin{equation}\label{dt}
	-b_h(\partial_t\bm{\theta_u},\theta_p) + s_h(\partial_t\theta_p,\theta_p) = 0.
\end{equation}
By choosing $\bm{\psi}_h = \partial_t\bm{\theta_u}$ in equation \eqref{errstokes1}, we obtain
\begin{equation*}
	\begin{aligned}
		\|\partial_t\bm{\theta_u}\|^2_{\bm{L}^2} &+ 
		\frac{1}{2}\frac{\partial}{\partial t}a_h(\bm{\theta_u,\theta_u}) + b_h(\partial_t\bm{\theta_u},\theta_p) = \left(\rho V - \rho_hV_h, \partial_t\bm{\theta_u}\right) 
		\\
		&+ \left(\left(\rho - \rho_h\right)\bm{\theta_u},\partial_t\bm{\theta_u}\right) - \left(\left(\rho - \rho_h\right)\bm{\eta_u},\partial_t\bm{\theta_u}\right) - \left(\left(\rho - \rho_h\right)\bm{u},\partial_t\bm{\theta_u}\right)
		\\
		& - \left(\rho\bm{\theta_u},\partial_t\bm{\theta_u}\right) + \left(\rho\bm{\eta_u},\partial_t\bm{\theta_u}\right) + \left(\partial_t\bm{\eta_u},\partial_t\bm{\theta_u}\right).
	\end{aligned}
\end{equation*}
A use of equation \eqref{dt} with the H\"older inequality, the Young's inequality, projection estimates, Lemma \ref{rhoh2} and estimate $\|\cdot\|_{L^\infty} \leq h^{-1}\|\cdot\|_{L^2}$ shows
\begin{equation*}
	\begin{aligned}
		\|\partial_t\bm{\theta_u}\|^2_{\bm{L}^2} + \frac{1}{2}\frac{\partial}{\partial t}&\left(a_h(\bm{\theta_u},\bm{\theta_u}) + s_h(\theta_p,\theta_p)\right) \leq C\left(h^{2k+2} + \|f - f_h\|^2_{0,\mathcal{T}_h} \right.
		\\
		&\left. \qquad\quad\qquad+\,\, h^{-2}\|f - f_h\|^2_{0,\mathcal{T}_h}\|\bm{\theta_u}\|^2_{\bm{L}^2} + \|\bm{\theta_u}\|^2_{\bm{L}^2} \right) + \frac{1}{2}\|\partial_t\bm{\theta_u}\|^2_{\bm{L}^2}.
	\end{aligned}
\end{equation*}
A kick-back argument with integration in time with respect to time using coercivity property \eqref{acoercive} and estimates from Theorem \ref{fi} completes the proof.
\end{proof}

\begin{rem}
Lemma \ref{utbound} yields a super-convergence result
\begin{equation*}
	\vertiii{\left(\bm{\Pi_uu} - \bm{u}_h\right)(t)} \leq Ch^{k+1} \quad \forall \quad t \in (0,T].
\end{equation*} 
Since for $1 \leq p < \infty$
\begin{equation*}
	\|\left(\bm{\Pi_uu} - \bm{u}_h\right)(t)\|_{\bm{L}^p} \leq C\vertiii{\left(\bm{\Pi_uu} - \bm{u}_h\right)(t)} \quad \forall \quad t \in (0,T]
\end{equation*}
then 
\begin{equation*}
	\|\left(\bm{\Pi_uu} - \bm{u}_h\right)(t)\|_{\bm{L}^p} \leq Ch^{k+1} \quad \forall \quad t \in (0,T].
\end{equation*}
Again as $\Omega_x \subset \R^2$, the discrete Sobolev imbedding (refer \cite{thomee2007galerkin}) implies
\begin{equation*}
	\begin{aligned}
		\|\left(\bm{\Pi_uu} - \bm{u}_h\right)(t)\|_{\bm{L}^\infty} &\leq C \left(log\left(\frac{1}{h}\right)\right)\vertiii{\left(\bm{\Pi_uu} - \bm{u}_h\right)(t)}
		\\
		&\leq C \left(log\left(\frac{1}{h}\right)\right)h^{k+1} \quad \forall \quad t \in (0,T].
	\end{aligned}
\end{equation*}
\end{rem}

\begin{rem}\label{ulpbound}
	If 
	\begin{equation*}
		\|\bm{u - \Pi_uu}\|_{L^\infty(0,T;\bm{L}^p)} \leq \left\{
		\begin{aligned}
			& Ch^{k+1}   \qquad\qquad\qquad\mbox{ for} \quad 1 \leq p < \infty,
			\\
			& C \left(log\left(\frac{1}{h}\right)\right)h^{k+1} \quad\mbox{for}\quad p = \infty,
		\end{aligned}
		\right.
	\end{equation*}
	then
	\begin{equation*}
		\|\bm{u} - \bm{u}_h\|_{L^\infty(0,T;\bm{L}^p)} \leq \left\{
		\begin{aligned}
			& Ch^{k+1}   \qquad\qquad\qquad\mbox{ for} \quad 1 \leq p < \infty,
			\\
			& C \left(log\left(\frac{1}{h}\right)\right)h^{k+1} \quad\mbox{for}\quad p = \infty. 
		\end{aligned}
		\right.
	\end{equation*}
\end{rem}


\begin{lem}\label{pbound}
	Under the hypothesis of Theorem \ref{fi},
 there exist a positive constant $C$ independent of $h$, such that for all $t \in (0,T]$
\begin{equation*}\label{thetapest}
	\begin{aligned}
		\int_{0}^{T}\vertiii{\left(\bm{\theta_u}(s),\theta_p(s)\right)}^2\,{\rm d}s \leq C h^{2(k+1)}.
	\end{aligned}
\end{equation*}	
\end{lem}

\begin{proof}
	Note that equation \eqref{errstokes1}-\eqref{errstokes2} is also equivalent to
	\begin{equation*}\label{errrstokes}
		\begin{aligned}
				\left(\frac{\partial \bm{\theta_u}}{\partial t},\bm{\psi}_h\right) &+ \tilde{\mathcal{A}}\left((\bm{\theta_u},\theta_p),(\bm{\psi}_h,w_h)\right) + \left(\rho\bm{\theta_u}, \bm{\psi}_h\right) = \left(\rho V - \rho_hV_h,\bm{\psi}_h\right) + \left(\partial_t\bm{\eta_u},\bm{\psi}_h\right)
				\\
				& +  \left(\rho\bm{\eta_u},\bm{\psi}_h\right) + \left(\left(\rho - \rho_h\right)\bm{\theta_u},\bm{\psi}_h\right) 
				- \left(\left(\rho - \rho_h\right)\bm{\eta_u},\bm{\psi}_h\right) + \left((\rho - \rho_h)\bm{u},\bm{\psi}_h\right), 
			\end{aligned}
	\end{equation*}
	where, $\tilde{\mathcal{A}}\left((\cdot,\cdot),(\cdot,\cdot)\right)$ is defined in \eqref{A_h}. By taking $\left(\partial_t\bm{\theta_u},\bm{\psi}_h\right)$ on the right hand side in above equation yields
\begin{equation*}
	\begin{aligned}
		\tilde{\mathcal{A}}&\left(\left(\bm{\theta_u},\theta_p\right),\left(\bm{\psi}_h,w_h\right)\right) = \left(\rho V - \rho_hV_h, \bm{\psi}_h\right) - \left(\left(\rho - \rho_h\right)\bm{\eta_u},\bm{\psi}_h\right) + \left(\left(\rho - \rho_h\right)\bm{\theta_u},\bm{\psi}_h\right)
		\\
		& \qquad\qquad- \left(\left(\rho - \rho_h\right)u, \bm{\psi}_h\right) + \left(\rho \bm{\eta_u},\bm{\psi}_h\right) - \left(\rho \bm{\theta_u},\bm{\psi}_h\right) + \left(\partial_t\bm{\eta_u},\bm{\psi}_h\right) - \left(\partial_t\bm{\theta_u},\bm{\psi}_h\right).
	\end{aligned}
\end{equation*}
A use of discrete inf-sup stability condition \eqref{dic1} with the H\"older inequality, projection estimate, Lemma \ref{rhoh2}, Lemma \ref{utbound} and Theorem \ref{fi} completes the proof.
\end{proof}

\begin{thm}\label{fi1}
Let $k \geq 1$ and let
\[
(\bm{u},p) \in \mathcal{C}^0([0,T];\bm{H}^{k+1}\cap \bm{W}^{1,\infty}) \times \mathcal{C}^0([0,T];H^k(\Omega_x)) 
\]
\[
f \in \mathcal{C}^1([0,T];H^{k+2}(\Omega)\cap W^{1,\infty}(\Omega)) 
\]
 be the solution of the Vlasov-Stokes equation \eqref{eq:continuous-model}-\eqref{contstokes}. Let 
 \[
 (\bm{u}_h,p_h,f_h) \in \mathcal{C}^0([0,T];\bm{H}_h\times L_h) \times \mathcal{C}^1([0,T];\mathcal{Z}_h)
 \]
  be the dG-dG approximation solution of \eqref{bh}-\eqref{ch}, then 
\begin{equation*}
	\|p(t) - p_h(t)\|_{L^2(0,T;L^2(\Omega_x))}  \leq Ch^{k} \quad \forall\,\,t \in [0,T]
\end{equation*}
where $C$ depends on the final time $T$, the polynomial degree $k$, the shape regularity of the partition and norm of $(\bm{u},p,f)$.
\end{thm}

\begin{proof}
A use of equation \eqref{ues} with triangle inequality gives
\begin{equation*}
	\begin{aligned}
		\|p - p_h\|_{0,\mathcal{T}^x_h} \leq \|\eta_p\|_{0,\mathcal{T}^x_h} + \|\theta_p\|_{0,\mathcal{T}^x_h}.
	\end{aligned}
\end{equation*}
Rest of the proof follows from Lemma \ref{uestimat} and Lemma \ref{pbound}.
\end{proof}

\begin{rem}\label{boundscase}
	From equation \eqref{eq:fhbd} and Theorem \ref{fi} along with equation \eqref{uL2}, we have
	\begin{equation*}
		f_h \in L^\infty(0,T;L^2(\mathcal{T}_h)) \quad \mbox{and} \quad \bm{u}_h \in L^\infty(0,T;\bm{L}^2).
	\end{equation*} 
Thanks to these bounds, the earlier local-in-time existence result for the discrete problem can be improved to a global-in-time existence result by extending the interval of existence.
\end{rem}

\subsection{Comment on $3$D}\label{3Dsubsec}

In $3$D the norm comparison inequality becomes
\begin{equation}\label{eq:3dnormcom}
	\|w_h\|_{L^p(T^x)} \leq Ch_x^{\frac{3}{p} - \frac{3}{q}}\|w_h\|_{L^q(T^x)}.
\end{equation} 
A use of the above inequality, yields
\begin{equation}
	\|\bm{u} - \bm{u}_h\|_{\bm{L}^\infty} \lesssim h_x\|\bm{u}\|_{\bm{W}^{1,\infty}} + h_x^{k-\frac{1}{2}}\|\bm{u}\|_{\bm{H}^{k+1}} + h_x^{-\frac{3}{2}}\|\bm{u} - \bm{u}_h\|_{\bm{L}^2}.
\end{equation}
On using \eqref{eq:3dnormcom} the result of Lemma \ref{ubound}, will modified as
\begin{equation}\label{thetauest3d}
	\begin{aligned}
		\frac12\frac{{\rm d}}{{\rm d}t}\|\bm{\theta_u}\|^2_{\bm{L}^2} & + \beta\vertiii{\bm{\theta_u}(t)}^2 + s_h(\theta_p,\theta_p) 
		\\
		&\leq C\left(h^{k+1} + \|f - f_h\|_{0,\mathcal{T}_h} + h^{-\frac{3}{2}}\|f - f_h\|_{0,\mathcal{T}_h}\|\bm{\theta_u}\|_{\bm{L}^2}\right)\|\bm{\theta_u}\|_{\bm{L}^2}.
	\end{aligned}
\end{equation}
The change in the estimate of $\mathcal{K}^2$ term \eqref{K2Estimate} is as follows:
\begin{equation}\label{K2Estimate3d}
	\begin{aligned}
		|\mathcal{K}^2| &\leq Ch^k\|\bm{u} - \bm{u}_h\|_{\bm{L}^\infty}\|f\|_{k+1,\Omega}\|\theta_f\|_{0,\mathcal{T}_h} 
		\\
		&\quad + Ch^{k+1}\left(\|\bm{u}\|_{\bm{W^}{1,\infty}}\|f\|_{k+1,\Omega} + \left(L + \|\bm{u}\|_{\bm{L}^\infty}\right)\|f\|_{k+2,\Omega}\right)\|\theta_f\|_{0,\mathcal{T}_h}
		\\
		&\leq Ch^k\left(h\|\bm{u}\|_{\bm{W}^{1,\infty}} + h^{k-\frac{1}{2}}\|\bm{u}\|_{\bm{H}^{k+1}} + h^{-\frac{3}{2}}\|\bm{u} - \bm{u}_h\|_{\bm{L}^2}\right)\|f\|_{k+1,\Omega}\|\theta_f\|_{0,\mathcal{T}_h} 
		\\
		&\quad + Ch^{k+1}\left(\|\bm{u}\|_{\bm{W}^{1,\infty}}\|f\|_{k+1,\Omega} + \left(L + \|\bm{u}\|_{\bm{L}^\infty}\right)\|f\|_{k+2,\Omega}\right)\|\theta_f\|_{0,\mathcal{T}_h}
		\\
		&\leq \left(C_{\bm{u},f} h^{k+1} + C h^{2k - \frac{1}{2}} + C_fh^{k-\frac{3}{2}}\|\bm{\theta_u}\|_{\bm{L}^2}\right)\|\theta_f\|_{0,\mathcal{T}_h}.
	\end{aligned}
\end{equation}
This will modify \eqref{thetafest} as:
\begin{equation}\label{thetafest3d}
	\begin{aligned}
		\frac{{\rm d}}{{\rm d}t}\|\theta_f\|_{0,\mathcal{T}_h} \leq C_{\bm{u},f}h^{k+1} + C  h^{2k-\frac{1}{2}} + Ch^{k -\frac{3}{2}}\|\bm{\theta_u}\|_{\bm{L}^2} + \|\theta_f\|_{0,\mathcal{T}_h}.
	\end{aligned}
\end{equation}
By adding \eqref{thetauest3d} and \eqref{thetafest3d}, we obtain
\begin{equation}
	\begin{aligned}
		\frac{{\rm d}}{{\rm d}t}\left(\|\bm{\theta_u}\|_{\bm{L}^2} + \|\theta_f\|_{0,\mathcal{T}_h}\right) &\lesssim h^{k+1} + h^{2k-\frac{1}{2}} + h^{k-\frac{3}{2}}\|\bm{\theta_u}\|_{\bm{L}^2} 
		\\
		&\quad+ \|\theta_f\|_{0,\mathcal{T}_h} + h^{-\frac{3}{2}}\|\bm{\theta_u}\|_{\bm{L}^2}\|\theta_f\|_{0,\mathcal{T}_h}.
	\end{aligned}
\end{equation}
Thus, for $k \geq 2$, it becomes
\begin{equation}
	\begin{aligned}
		\frac{{\rm d}}{{\rm d}t}\left(\|\bm{\theta_u}\|_{\bm{L}^2} + \|\theta_f\|_{0,\mathcal{T}_h}\right) &\lesssim h^{k+1} + \left(\|\bm{\theta_u}\|_{\bm{L}^2} + \|\theta_f\|_{0,\mathcal{T}_h}\right) + h^{-\frac{3}{2}}\left(\|\bm{\theta_u}\|^2_{\bm{L}^2} + \|\theta_f\|^2_{0,\mathcal{T}_h}\right)
		\\
		&\lesssim h^{k+1}+ \vertiii{(\bm{\theta_u},\theta_f)} + h^{-\frac{3}{2}}\vertiii{(\bm{\theta_u},\theta_f)}^2,
	\end{aligned}
\end{equation}
with same definition of $\vertiii{(\cdot,\cdot)}$ as in Theorem \ref{fl2}. Set 
\[
\Psi(t) = h^{k+1} + \int_0^t\left(\vertiii{(\bm{\theta_u},\theta_f)(s)} + h^{-\frac{3}{2}}\vertiii{(\bm{\theta_u},\theta_f)(s)}^2\right){\rm d}s.
\] 
On integration and proceed in a similar manner as in the proof of Theorem \ref{fl2} to obtain 
\[
\Psi(t)\left(1 - h^{k - \frac{1}{2}}(e^{CT} - 1)\right) \leq e^{CT}h^{k+1}. 
\]
For $k \geq 2$ and with smallness assumption on $h$
\[
(1 - h^{k-\frac{1}{2}}(e^{cT} - 1)) > 0.
\]
Hence, the result of Theorem \ref{fi} follows for $3$D for $k \geq 2$. 

\section{Numerical Experiments}
This section, we report on some numerical simulations based on a splitting algorithm. We have approximated solutions to the following Vlasov-unsteady Stokes model with source terms:
\begin{equation}\label{eq:continuous-models-ss12}
	\left\{
	\begin{aligned}
		\partial_t f + v\cdot \nabla_x f + \nabla_v \cdot \Big( \left( \bm{u} - v \right) f \Big) & = \mathcal{F}(t,x,v)  \quad\mbox{ in }(0,T)\times\Omega_x\times\Omega_v,
		\\
		f(0,x,v) & = f_{0}(x,v)\qquad  \mbox{ in }\Omega_x\times\Omega_v.
	\end{aligned}
	\right.
\end{equation}
\begin{equation*}\label{contstokess-ss12}
	\left\{
	\begin{aligned}
		\partial_t\bm{u} - \Delta_x \bm{u} + \rho \bm{u} +\nabla_x p & = \rho V + \mathcal{G}(t,x)  \quad \mbox{ in }\Omega_x,
		\\
		\nabla_x \cdot \bm{u} & = 0  \qquad  \mbox{ in }\Omega_x.
	\end{aligned}
	\right.
\end{equation*}
For the computational results, we employed splitting algorithm for the linear kinetic equation \eqref{eq:continuous-models-ss12}, using the Lie-Trotter splitting method. 
To achieve this, firstly, split the equation \eqref{eq:continuous-models-ss12} as:
\begin{equation}\label{sEq:ucontstokes-ss12}
	(a)\quad  \begin{aligned}
		\partial_t f + \nabla_v\cdot\left(\left( \bm{u} - v \right)f \right) = \mathcal{F}(t,x,v).
	\end{aligned}
	\qquad (b) \quad	\begin{aligned}
		\partial_t f + v\cdot\nabla_x f = 0,
	\end{aligned}   
\end{equation}
To compute the solution $f$, we first solve part $(a)$ of \eqref{sEq:ucontstokes-ss12} over the full time step using $f_0$ as the initial data to obtain an intermediate solution $\tilde{f}$. Next, part $(b)$ of \eqref{sEq:ucontstokes-ss12} is solved over the full time step with $\tilde{f}$ as the initial data.

For a temporal discretization, 
let $\{t_n\}_{n=0}^{N}$ be a uniform partition of the time interval $[0,T]$, where $t_n = n\Delta t$ with time step $\Delta t > 0$.
Let $F^n \in \mathcal{Z}_h, \bm{U}^n \in \bm{H}_h, P^n \in L_h, \rho_h^n$ and $\rho_h^nV_h^n$ be the approximations of $f^n = f(t_n), \bm{u}^n = \bm{u}(t_n), p^n = p(t_n), \rho^n = \rho(t_n)$ and $\rho^nV^n = \rho(t_n)V(t_n)$, respectively. 
Our numerical algorithm is to find $\left( \bm{U}^{n+1},P^{n+1},F^{n+1}\right) \in \bm{H}_h \times L_h \times \mathcal{Z}_h$, for $n = 0,1,\cdots, N-1$ such that 
\begin{equation}\label{fdsvusch1-ss12}
	\left\{
	\begin{aligned}
		\left(\frac{\bm{U}^{n+1} - \bm{U}^n}{\Delta t}, \bm{\psi}_h\right) + a_h(\bm{U}^{n+1},\bm{\psi}_h) &+ b_h(\bm{\psi}_h,P^{n+1}) + \left(\rho_h^n\bm{U}^{n+1}, \bm{\psi}_h\right)
		\\
		&= \left(\rho_h^nV_h^n + \mathcal{G}^{n+1} , \bm{\psi}_h\right) ~ \forall ~ \bm{\psi}_h \in \bm{H}_h,
		\\
		- b_h(\bm{U}^{n+1},w_h) + s_h(P^{n+1},w_h) &= 0 \quad \forall \quad w_h \in L_h,
		\\
		\left(\frac{\tilde{F} - F^n}{\Delta t},\phi_h\right) + \mathcal{B}^x_h\left(\bm{U}^{n+1};F^n,\phi_h\right) 
		&= \left(\mathcal{F}^n,\phi_h\right) \quad \forall \quad \phi_h \in \mathcal{Z}_h,
		\\
		\left(\frac{F^{n+1} - \tilde{F}}{\Delta t}, \psi_h\right) +\mathcal{B}^v_h(\tilde{F},\psi_h) &= 0 \quad \forall \quad \psi_h \in \mathcal{Z}_h,
	\end{aligned}
	\right.
\end{equation}
where, $ a_h(\bm{U}^n,\bm{\psi}_h), b_h(\bm{\psi}_h,P^n)$ and $s_h(P^n,w_h)$ are defined by \eqref{aih}-\eqref{sh} at $t = t_n$. In \eqref{fdsvusch1-ss12} we have  used the following notations:
\begin{equation*}
	\mathcal{B}^v_h(F^n,\phi_h) := \sum_{R \in \mathcal{T}_h}\mathcal{B}^v_{h,R}(F^n,\phi_h)
\end{equation*}
with
\begin{equation*}
	\begin{aligned}
		\mathcal{B}^v_{h,R}(F^n,\phi_{h}) & :=  - \int_{T^v}\int_{T^x}F^n\,v.\nabla_x\phi_h \, {\rm d}x\,{\rm  d}v + \int_{T^v}\int_{\partial T^x}\reallywidehat{v\cdot\bm{n}F^n}\phi_{h}\, {\rm d}s(x)\,{\rm d}v
	\end{aligned}
\end{equation*}
and 
\begin{equation*}
	\mathcal{B}^x_h(\bm{U}^{n+1};F^n,\phi_h) := \sum_{R \in \mathcal{T}_h}\mathcal{B}^x_{h,R}(\bm{U}^{n+1};F^n,\phi_h),
\end{equation*}
with
\begin{equation*}
	\begin{aligned}
		\mathcal{B}^x_{h,R}(\bm{U}^{n+1};F^n,\phi_h) & := - \int_{T^x}\int_{T^v}F^n\left(\bm{U}^{n+1} - v\right).\nabla_v\phi_h \, {\rm d}v\,{\rm d}x 
		\\&\qquad+ \int_{T^x}\int_{\partial T^v} \reallywidehat{\left(\bm{U}^{n+1} - v\right)\cdot\bm{n}F^n}\phi_h \, {\rm d}s(v)\,{\rm d}x
	\end{aligned}
\end{equation*} 
wherein the numerical fluxes are defined by \eqref{flux} at $t = t_n$.

For our computation, let $\bm{x}_l$ and $\bm{v}_m$ are the nodes in $\mathcal{T}_h^x$ and $\mathcal{T}_h^v$, respectively, where $l = 1, \cdots, N_x$ and  $m = 1, \cdots, N_v$. Any function in the $\mathcal{Z}_h$ space can be represented as 
\[
g = \sum_{l,m}g(\bm{x}_l,\bm{v}_m)L_x^l(\bm{x})L_v^m(\bm{v})
\]
on $R$, where $L_x^l(\bm{x})$ and $L_v^m(\bm{v})$ are the $l$-th and $m$-th Lagrangian interpolating polynomials in $T^x$ and $T^v$, respectively.  

In this setting, we can solve the equations for $f$ in \eqref{sEq:ucontstokes-ss12} in the reduced dimensions. For example, in equation $\eqref{sEq:ucontstokes-ss12} (a)$, we fix a nodal point in the $x$-direction, say $\bm{x}_l$, and solve
\begin{equation*}
	\begin{aligned}
		\partial_t f(\bm{x}_l) + \nabla_v\cdot\left(\left(\bm{u}(\bm{x}_l) - v\right)f(\bm{x}_l)\right) = \mathcal{F}(t,\bm{x}_l,v)
	\end{aligned}
\end{equation*}
in the $v$-direction and obtain an update of point values of $f(\bm{x}_l,\bm{v}_m)$ for all $\bm{v}_m \in \mathcal{T}^v_{h}$.

Similarly, for equation $\eqref{sEq:ucontstokes-ss12} (b)$. We fix a nodal point in the $v$-direction, say $\bm{v}_m$, and solve 
\[
\partial_tf(\bm{v}_m) + \bm{v}_m\cdot\nabla_x f(\bm{v}_m) = 0
\]
by a dG method in the $x$-direction and obtain an update of point values of $f(\bm{x}_l,\bm{v}_m)$ for all $\bm{x}_l \in \mathcal{T}^x_{h}$.

For the plots, we denotes the degree of polynomials in $x$ and $v$-variables by $k_x$ and $k_v$, respectively. The mesh sizes for $\mathcal{T}_h^x$ and $\mathcal{T}_h^v$ are represented by $h_x$ and $h_v$, respectively. For the numerical experiments, we take $h_x = h_v = h$. We calculate the errors $f - f_h, \bm{u - u_h}$ and $p - p_h$ in $L^2(\Omega), \bm{L^2}$ and $L^2(\Omega_x)$-norms, respectively, at final time $T$ and denoted by errL2f, errL2u and errL2p, respectively.

\begin{exm}\label{exm-1-vus}
	The first example for which we test the proposed scheme has the following as its exact solution:
	
	\begin{equation*}
		\begin{aligned}
			f(t,x,y,v_1,v_2) &= \cos(t)\sin(2\pi x)\sin(2\pi y)e^{(-v_1^2-v_2^2)}(1-v_1^2)(1-v_2^2)(1+v_1)(1+v_2),
			\\
			u_1(t,x,y) &= \cos(t)\left(-\cos(2\pi x)\sin(2\pi y) + \sin(2\pi y)\right),
			\\
			u_2(t,x,y) &= \cos(t)\left(\sin(2\pi x)\cos(2\pi y) - \sin(2\pi x)\right),
			\\
			p(t,x,y) &= 2\pi\cos(t)\left(\cos(2\pi y) - \cos(2\pi x)\right).
		\end{aligned}
	\end{equation*} 
	Note that corresponding is the initial data
	\[
	f(0,x,y,v_1,v_2) = \sin(2\pi x)\sin(2\pi y)e^{(-v_1^2-v_2^2)}(1-v_1^2)(1-v_2^2)(1+v_1)(1+v_2),
	\] 
	\[
	u_1(0,x,y) = -\cos(2\pi x)\sin(2\pi y) + \sin(2\pi y), \qquad
	u_2(0,x,y) = \sin(2\pi x)\cos(2\pi y) - \sin(2\pi x).
	\]
\end{exm}	 

We run the simulations for the domains $\Omega_x = [0,1]^2, \Omega_v = [-1,1]^2$, the penalty parameter as $10$ and for $k_x = 1,1$ and $k_v = 1,2$ with the final time $0.1$ and for $k_x = k_v = 2$ the final time is $0.01$.

\begin{figure}
	\centering
	\includegraphics[width=6.25cm]{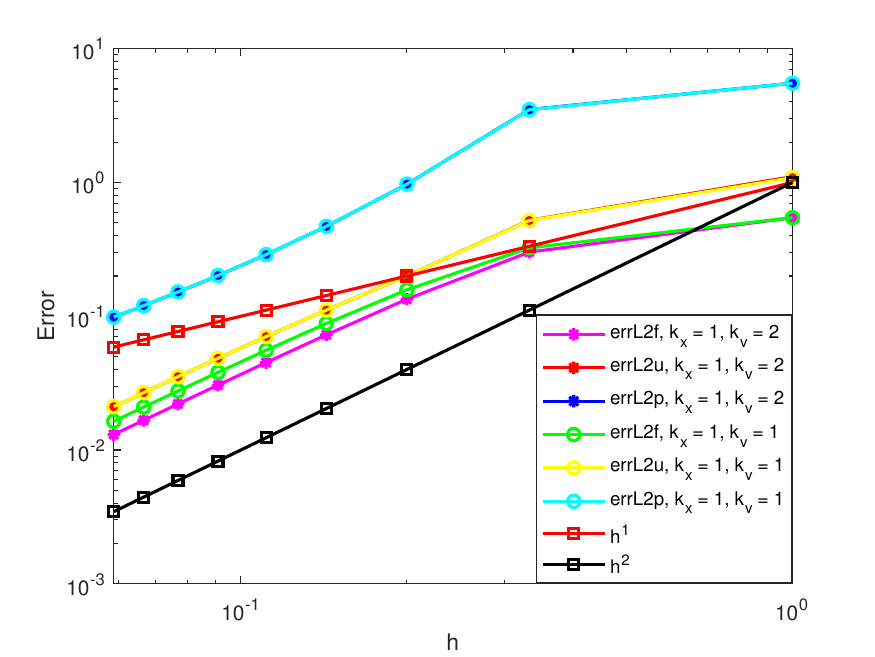}
	\includegraphics[width=6.25cm]{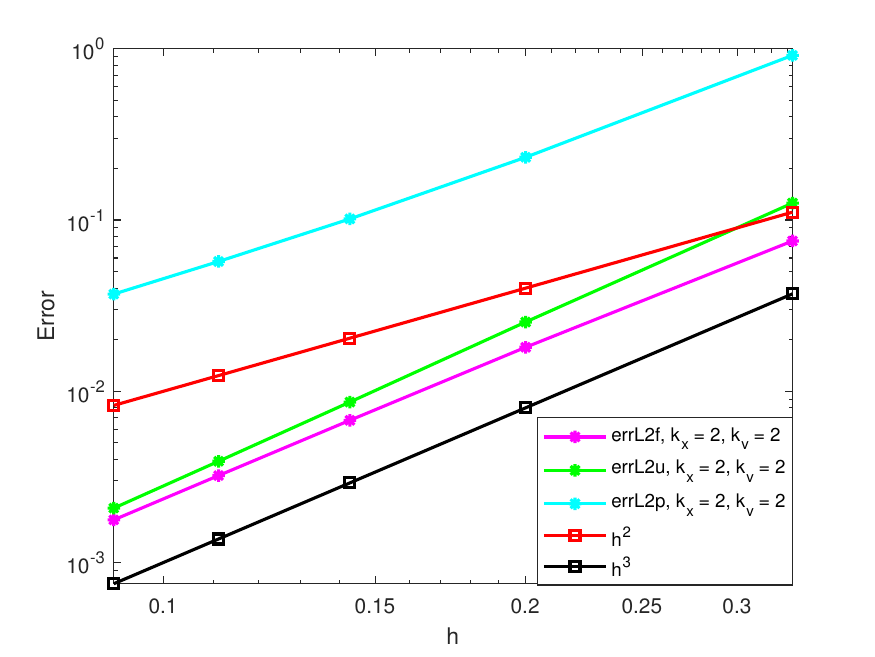}
	\caption{Convergence rates for the distribution function $f$, for the velocity $\bm{u}$ and for the pressure $p$ in the Example \ref{exm-1-vus}.   }
	\label{fig:1-vus}
\end{figure}

\begin{exm}\label{exm-2-vus}
	The second example takes the following as the exact solutions:
	\begin{equation*}
		\begin{aligned}
			f(t,x,y,v_1,v_2) &= \sin(\pi(x-t))\sin(\pi(y-t))e^{(-v_1^2-v_2^2)}(1-v_1^2)(1-v_2^2)(1+v_1)(1+v_2), 
			\\
			u_1(t,x,y) &= -\cos(2\pi x - t)\sin(2\pi y - t), \qquad u_2(t,x,y) = \sin(2\pi x - t)\cos(2\pi y - t),
			\\
			p(t.x,y) &= 2\pi\left(\cos(2\pi y - t) - \cos(2\pi x - t)\right),
		\end{aligned}
	\end{equation*} 
	with the corresponding initial data:
	\[
	f(0,x,y,v_1,v_2) = \sin(\pi x)\sin(\pi y)e^{(-v_1^2-v_2^2)}(1-v_1^2)(1-v_2^2)(1+v_1)(1+v_2),
	\]
	\[
	u_1(0,x,y) = -\cos(2\pi x)\sin(2\pi y), \qquad u_2(0,x,y) = \sin(2\pi x)\cos(2\pi y). 
	\]
\end{exm}

We run the simulations for the domains $\Omega_x = [0,1]^2$ and $\Omega_v = [-1,1]^2$. The penalty parameter chosen to be $10$ and for $k_x = 1,1$ and $k_v = 1,2$ with the final time $0.1$ and for $k_x = k_v = 2$ the final time is $0.01$.

\begin{figure}
	\centering
	\includegraphics[width=6.25cm]{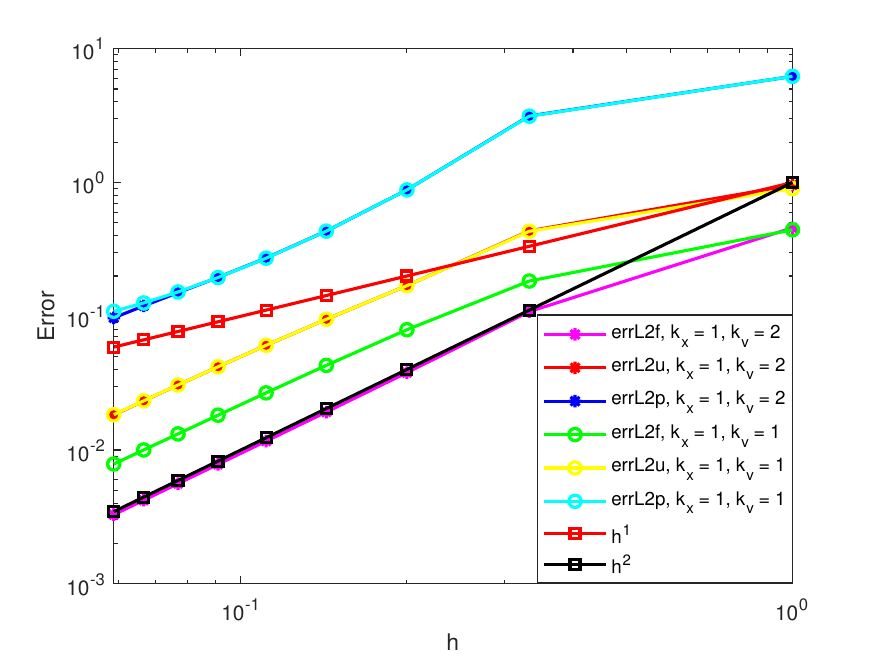}
	\includegraphics[width=6.25cm]{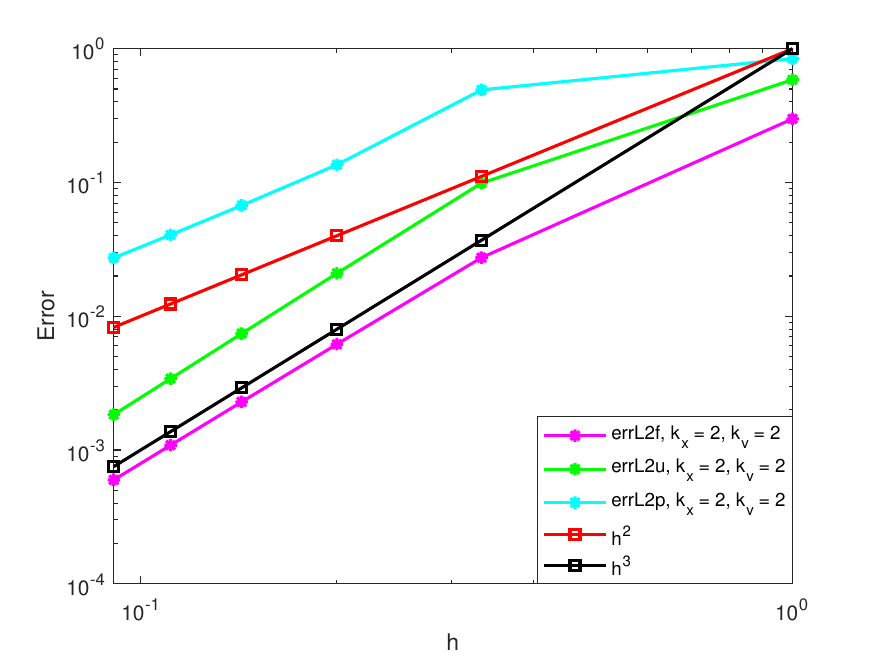}
	\caption{Convergence rates for the distribution function $f$, for the velocity $\bm{u}$ and for the pressure $p$ in the Example \ref{exm-2-vus}.   }
	\label{fig:2-vus}
\end{figure}  
\noindent
\textbf{Observations:}
\begin{itemize}
	\item From Figure \ref{fig:1-vus}(Left) and Figure \ref{fig:2-vus}(Left), it is easily observed that by choosing degrees of polynomials $k_x = 1, k_v = 1$ and $k_x = 1, k_v = 2$, we achieve an order of convergence for the distribution function $f$ and the fluid velocity $\bm{u}$ equalling $\min{(k_x,k_v)} + 1$ which is $2$ and for pressure $p$, the order of convergence obtained equals $\min(k_x,k_v)$ which is $1$. 
	\item By taking degree of polynomials $k_x = k_v = 2$, from Figure \ref{fig:1-vus}(Right) and Figure \ref{fig:2-vus}(Right), we obtain the order of convergence for the distribution function $f$ and the fluid velocity $\bm{u}$ equals $\min(k_x,k_v) + 1$ which is $3$, for the fluid pressure $p$ the order of convergence equals $\min(k_x,k_v)$ which is $2$.
\end{itemize}

	\section{Conclusion}\label{sec:conclude}

In this paper, a semi-discrete numerical method for $2$D Vlasov-Stokes equation is introduced and analysed (see \eqref{bh}-\eqref{ch2} for the discrete problem). This is a dG-dG method for the Vlasov and the Stokes equations in phase space. The scheme is mass and momentum conserving. In continuum case, a non-negative initial data $(f_0(x,v) \geq 0)$ yields a non-negative solution for all times. At present, we are unable to prove a similar positivity preserving property for our discrete system and hence, it is difficult to prove a discrete energy dissipation property. Optimal rates of convergence, with regards to the degree of polynomials $k \geq 1$, for the distribution function $f(t,x,v)$ and the fluid velocity $\bm{u}(t,x)$ are proved in the Theorem \ref{fi} in $L^\infty(0,T;L^2(\Omega))$ and $L^\infty(0,T;\bm{L}^2)$-norms, respectively. These optimal rates help us deduce super-convergence result for $\left(\bm{\Pi_uu} - \bm{u}_h\right)$ in $\vertiii{\cdot}$-norm (see Lemma \ref{utbound}). This enables us to derive optimal rate of convergence for the fluid velocity in $\L^\infty(0,T;\bm{L}^p)$-norm $1 \leq p \leq \infty$, subject to the availability of certain projection estimates. In Lemma \ref{pbound}, super-convergence result for $\left(\bm{\Pi_uu} - \bm{u}_h\right)$ and $\left(\Pi_pp - p_h\right)$ in $L^2(0,T;\vertiii{\cdot})$ and $L^2((0,T)\times \Omega_x)$-norms, respectively are shown. The optimal rate of convergence for the fluid pressure in $L^2((0,T)\times \Omega_x)$-norm is derived in Theorem \ref{fi1}. Finally, in subsection \ref{3Dsubsec} we comment on the $3$D Vlasov-Stokes equation.\\

\section*{Statements and Declarations}
\subsection*{Ethical Approval} Not Applicable
\subsection*{Data Availability}
The codes during the current study are available from the corresponding author on reasonable request.
\subsection*{Conflict of Interest}
The authors declare that they have no conflict of interest.
\subsection*{Funding}
The work of the second author was supported by the SERB project Grant No. CRG/2023/000721 of Govt. of India.

\subsection*{Author Contributions}
All authors contributed equally to preparing this manuscript. All authors read and approved the final manuscript.

\subsection*{Acknowledgments} 
The second author is gratefully acknowledges the support by the SERB project Grant No. CRG/2023/000721 of Govt. of India.

\bibliography{references}
\bibliographystyle{amsalpha}

\end{document}